\documentclass[a4paper,12pt,leqno]{amsart} 
\usepackage{amsmath,amsthm,amssymb}  

\usepackage{graphicx}

\numberwithin{equation}{section}

\theoremstyle{plain}
\newtheorem{theorem}{Theorem}[section]
\newtheorem{proposition}[theorem]{Proposition}         
\newtheorem{corollary}[theorem]{Corollary} 
\newtheorem{lemma}[theorem]{Lemma} 
\newtheorem{definition}[theorem]{Definition}  

\theoremstyle{definition}  
 
\newtheorem{remark}[theorem]{Remark} 

\newcommand{\C}{\mathbb C}   
\newcommand{\R}{\mathbb R}
\newcommand{\Z}{\mathbb Z}
\newcommand{\N}{\mathbb N}  

\newcommand{\al}{\alpha}
\newcommand{\be}{\beta} 
\newcommand{\ga}{\gamma}
\newcommand{\de}{\delta}
\newcommand{\la}{\lambda}
\newcommand{\si}{\sigma} 

\newcommand{\eps}{\epsilon}
\newcommand{\Om}{\Omega}
\newcommand{\De}{\Delta}
\renewcommand{\th}{\theta}
\newcommand{\om}{\omega}

\newcommand{\Ga}{\Gamma}
\newcommand{\ze}{\zeta}

\DeclareMathOperator{\diag}{diag}
\DeclareMathOperator{\ar}{arg}

\newcommand{\no}{\noindent}
      
\newcommand{\pr}{\prime} 
\newcommand{\prr}{{\prime\prime}} 
\newcommand{\st}{\ \vert\ }   
\renewcommand{\ll}{\lq\lq}
\newcommand{\rr}{\rq\rq\ }
\newcommand{\rrr}{\rq\rq}

\newcommand{\bp}{\begin{pmatrix}} 
\newcommand{\ep}{\end{pmatrix}} 
\newcommand{\bsp}{\left(\begin{smallmatrix}} 
\newcommand{\esp}{\end{smallmatrix}\right)}

\newcommand{\zbar}{  {\bar z}  }
\newcommand{\zzb}{ {z\bar z}  }

\renewcommand{\i}{ {\scriptscriptstyle\sqrt{-1}}\, }
\newcommand{\ii}{ {\scriptstyle\sqrt{-1}}\, }

\newcommand{\Psiz}{  \Psi^{(0)}  }
\newcommand{\Psii}{  \Psi^{(\infty)}  }
\newcommand{\psiz}{  \psi^{(0)}  }
\newcommand{\psii}{  \psi^{(\infty)}  }
\newcommand{\Sz}{  S^{(0)}  }
\newcommand{\Si}{  S^{(\infty)}  }
\newcommand{\Qz}{  Q^{(0)}  }
\newcommand{\Qi}{  Q^{(\infty)}  }
\newcommand{\barQz}{  {\bar Q}^{(0)}  }
\newcommand{\barQi}{  {\bar Q}^{(\infty)}  }
\newcommand{\Omz}{  \Om^{(0)}  }
\newcommand{\Omi}{  \Om^{(\infty)}  }
\newcommand{\dz}{  d_{(0)}  }
\newcommand{\di}{  d_{(\infty)}  }

\newcommand{\tPsiz}{  \tilde\Psi^{(0)}  }
\newcommand{\tPsii}{  \tilde\Psi^{(\infty)}  }
\newcommand{\tpsiz}{  \tilde\psi^{(0)}  }
\newcommand{\tpsii}{  \tilde\psi^{(\infty)}  }

\newcommand{\tSi}{ \tilde S^{(\infty)}  }
\newcommand{\tQz}{  \tilde Q^{(0)}  }
\newcommand{\tQi}{  \tilde Q^{(\infty)}  }
\newcommand{\tPi}{  \tilde\Pi  }

\newcommand{\fPsi}{  \,\underline{\!\Psi\!}\, }
\newcommand{\fPsiz}{  \,\underline{\!\Psi\!}\,^{(0)}  }
\newcommand{\fPsii}{  \,\underline{\!\Psi\!}\,^{(\infty)}  }
\newcommand{\fOmz}{  \,\underline{\!\Om\!}\,^{(0)}  }
\newcommand{\fOmi}{  \,\underline{\!\Om\!}\,^{(\infty)}  }

\newcommand{\boxe}{ \boxed{\vphantom{.}} }
\newcommand{\boxep}{ \boxed{\pi\vphantom{.}} }

\newcommand{\epr}{  e^\prime  }
\newcommand{\eprr}{  e^{\prime\prime}  }

\newcommand{\s}{s_m}

\begin{document}     

\title[Isomonodromy aspects II: Riemann-Hilbert problem]{Isomonodromy aspects of the tt*
equations of Cecotti and Vafa 
\\
II. Riemann-Hilbert problem}  

\author{Martin A. Guest, Alexander R. Its, and Chang-Shou Lin}      

\date{}   

\begin{abstract} In \cite{GuItLiXX} (part I) we computed the Stokes data, though not the \ll connection matrix\rrr, for the smooth solutions of the tt*-Toda equations whose existence we established by p.d.e.\ methods.  Here we give an alternative proof of the existence of some of these solutions by solving a Riemann-Hilbert problem.  In the process, we compute the connection matrix for all smooth solutions, thus completing the computation of the monodromy data.  We also give connection formulae relating the asymptotics at zero and infinity of all smooth solutions, clarifying the region of validity of the formulae 
established earlier by Tracy and Widom.  Finally, for the tt*-Toda equations, we resolve some conjectures of  Cecotti and Vafa concerning the positivity of  $S+S^t$ (where $S$ is the Stokes matrix) and the unimodularity of the eigenvalues of the monodromy matrix. 
\end{abstract}


\maketitle 

\section{Introduction}\label{intro}

The tt*  (topological---anti-topological fusion) equations were introduced by Cecotti and Vafa to describe certain
deformations of supersymmetric quantum field theories (\cite{CeVa91},\cite{CeVa92},\cite{CeVa92a}).  They are nonlinear equations and difficult to study directly.  On the other hand they are deeply related to geometry as well as physics, and they can be formulated as isomonodromic deformation equations (\cite{Du93}), which allows one to study them from the point of view of integrable systems.

A special case of the tt* equations,  which we call the  tt*-Toda equations, is
\begin{equation}\label{ost}
 2(w_i)_{\zzb}=-e^{2(w_{i+1}-w_{i})} + e^{2(w_{i}-w_{i-1})}, \ \  
 w_i:U\to\R
\end{equation}
subject to the \ll anti-symmetry condition\rr
\begin{equation}\label{as}
\begin{cases}
\ \  w_0+w_{l-1}=0, \ w_1+w_{l-2}=0,\ \ \dots\\
\ \  w_l+w_{n}=0, \ w_{l+1}+w_{n-1}=0,\ \ \dots
 \end{cases}
\end{equation}
for some $l\in\{0,\dots,n+1\}$,  and the radial condition $w_i=w_i(\vert z\vert)$. 
Here, $U$ is some open subset of $\C=\R^2$, and we assume that $w_i=w_{i+n+1}$ for all $i\in\Z$ and  $w_0+\cdots+w_n=0$.   Equation (\ref{ost}) alone is 
the well known two-dimensional periodic Toda lattice but with \ll opposite sign\rrr.

For $l=0$ (equivalently $l=n+1$),
condition (\ref{as}) means that
$w_i+w_{n-i}=0$ for all $i$.  This case was introduced by Cecotti and Vafa in \cite{CeVa91} as a relatively simple example of the tt* equations.  They predicted the existence of special solutions related to the quantum cohomology of complex projective spaces $\C P^n$ and unfoldings of singularities of type $A_{n+1}$.  For $n=1$ or $2$, the functions $w_0,\dots,w_n$ reduce to a single function, and equations (\ref{ost}) reduce to the sinh-Gordon or Tzitzeica (Bullough-Dodd) equations, which have been well studied (see \cite{MTW77},\cite{Ki89} and the references in \cite{FIKN06}).  However, the lack of
mathematical results for $n\ge 3$ has hindered subsequent attempts to verify the intriguing conjectures made by Cecotti and Vafa.  Our project addresses this problem.

 In \cite{GuLiXX},\cite{GuItLiXX}, using quite elementary p.d.e.\ arguments,
we found all solutions of the tt*-Toda equations on $U=\C^\ast=\C-\{0\}$ for $3\le n\le 5$, and we
computed the corresponding Stokes data.  
Their geometrical interpretation --- such as quantum cohomology or singularity theory --- in the case where the Stokes data is integral, was studied in \cite{GuLi2XX}.   
In this article we shall use the Riemann-Hilbert method of \cite{FIKN06}
to investigate these solutions more explicitly.

To explain this point of view, let us recall from \cite{GuItLiXX} that
equation (\ref{ost}) is the
compatibility condition $2w_{z\zbar}=[W^t,W]$ for the linear system
\begin{align*}
\begin{cases}
\Psi_z&=(w_z+\tfrac1\la W)\Psi\\
\Psi_\zbar&=(-w_{\zbar}+\la W^t)\Psi
\end{cases}
\end{align*}
where
\[
 w=\diag(w_0,\dots,w_n),\ \ 
 W=
 \left(
\begin{array}{c|c|c|c}
\vphantom{(w_0)_{(w_0)}^{(w_0)}}
 & \!e^{w_1\!-\!w_0}\! & &  
 \\
\hline
  &  & \  \ddots\   & \\
\hline
\vphantom{\ddots}
  & &  &  e^{w_n\!-\!w_{n\!-\!1}}\!\!\!
\\
\hline
\vphantom{(w_0)_{(w_0)}^{(w_0)}}
\!\! e^{w_0\!-\!w_n} \!\!  & &  &  \!
\end{array}
\right).
\]
Putting $x=\vert z\vert$, the  radial version
$(xw_{x})_x=2x[W^t,W]$ of (\ref{ost}) is the  compatibility condition of another linear system
\begin{equation}\label{zcc}
\begin{cases}
\Psi_\mu&=\left( -\tfrac{1}{\mu^2} xW - \tfrac{1}{\mu} xw_x + xW^t\right) \Psi\\
\Psi_x&=\left(  \tfrac1\mu W + \mu W^t\right)\Psi
\end{cases}
\end{equation}
where $\mu=\la x/z$.
The $\mu$-system here is a meromorphic linear ordinary differential equation in the complex variable $\mu$ with poles of order two at $\mu=0$ and $\mu=\infty$.  Solutions of 
$(xw_{x})_x=2x[W^t,W]$ can be interpreted as isomonodromic $x$-deformations of this system,  i.e.\ for which the monodromy data of the $\mu$-system is independent of $x$.  We refer to
\cite{GuItLiXX} for further explanation of these assertions.  

For later convenience we introduce the new variable
\[
\zeta=\la/z=\mu/x
\]
and rewrite the $\mu$-system as the following $\ze$-system:
\begin{equation}\label{ode}
\Psi_\ze =  \left( -\tfrac{1}{\ze^2} W - \tfrac{1}{\ze} xw_x + x^2W^t \right) \Psi.
\end{equation}
The monodromy data consists of collections of Stokes matrices $S^{(0)}_i$, $S^{(\infty)}_i$ at the poles $\ze=0$, $\ze=\infty$ (relating solutions on different Stokes sectors) and a connection matrix $E$ (which relates solutions near $0$ with solutions near $\infty$). 
Assigning the monodromy data to a meromorphic connection is known as the \ll direct monodromy transform\rrr.  In \cite{GuItLiXX} we computed the Stokes matrices explicitly in terms of the asymptotic data at $z=0$.

For the \ll inverse monodromy transform\rr one seeks to recover the original connection from its monodromy data, and it is well known that this can be formulated as a Riemann-Hilbert problem, which, in turn,  is equivalent to the problem of solving a linear integral equation.  There is no guarantee that such a problem can be solved easily, but the linear equation is more tractable than the original nonlinear equation, and the Riemann-Hilbert formulation is useful for obtaining asymptotic properties of solutions.

Our main achievement in this article is to formulate precisely a suitable Riemann-Hilbert problem and prove its solvability for certain monodromy data.  As a consequence, we obtain a new proof of the existence of smooth radial solutions of (\ref{ost}),({\ref{as}) on $\C^\ast$, at least for an open subset of the Stokes data.  Moreover, for all smooth radial solutions we find the corresponding connection matrices, thus completing the explicit computation of the monodromy data.   As one might expect, the connection matrix is particularly simple for these solutions.

Let us mention two further applications of this approach.

\no{\em Asymptotics: }
Consider smooth radial solutions of (\ref{ost}),({\ref{as}) on $\C^\ast$.  These must satisfy the boundary  conditions
\[
\lim_{z\to 0}  \ \frac{w_i(z)}{\log\vert z\vert}=\text{constant,}
\quad
\lim_{z\to \infty}  \ w_i(z)=0
\]
(see Appendix C). 
In \cite{GuLiXX},\cite{GuItLiXX}, we prove that these $n+1$ constants determine $w_0,\dots,w_n$ uniquely. 
It is easy to see that $w_i(z)$ decays exponentially at $\infty$, but the explicit form of this dependence is less obvious.  
For example, when $n=3$ and $l=4$, equations
(\ref{ost}),({\ref{as}) reduce to  
\begin{equation}\label{ost2}
\begin{cases}
u_{z\zbar}&= \ e^{2u} - e^{v-u} 
\\
v_{z\zbar}&= \ e^{v-u} - e^{-2v}
\end{cases}
\end{equation}
where $u=2w_0$, $v=2w_1$.  At infinity, these equations can be written in the simple form
\begin{equation}\label{ost2linear}
\begin{cases}
u_{z\zbar}&= \ 3u-v+\ \text{higher order terms}
\\
v_{z\zbar}&= \ 3v-u+\ \text{higher order terms}.
\end{cases}
\end{equation}
It is not difficult to deduce asymptotic expressions of the form
\[
u(z)\sim A \vert z\vert^{-\frac12} e^{-2\sqrt2\vert z\vert},
\quad
v(z)\sim B \vert z\vert^{-\frac12} e^{-2\sqrt2\vert z\vert}
\]
as $\vert z\vert \to\infty$.  Since we have $u(z)\sim\ga \log \vert z\vert$, 
$v(z)\sim\de \log \vert z\vert$ as $\vert z\vert \to 0$, the question arises of finding an explicit relation or \ll connection formula\rr  between
$A,B$ and $\ga,\de$.    From the purely analytic viewpoint of nonlinear p.d.e.\
this seems very difficult. 

As well as providing another approach to the p.d.e.\ existence results, the Riemann-Hilbert method answers this question (Theorem \ref{asymptoticsatinfinity}, Corollary \ref{ambiguity}). 
We note that the first connection formulae of this type were obtained by Tracy and Widom in \cite{TrWi98}.

\no{\em Monodromy criteria for smooth solutions: }
We have already explained in \cite{GuLiXX},\cite{GuLi2XX},\cite{GuItLiXX} how the tt*-Toda equations are related to geometry, in particular harmonic maps and quantum cohomology.  Here we comment on another aspect, the theory of variations of TERP structure developed by Hertling, Sabbah, and Sevenheck (\cite{He03},\cite{HeSe07},\cite{Sa08},\cite{HeSa11}).  This provides a general framework, motivated by singularity theory, for discussing criteria involving monodromy data for various properties of solutions of the equations.

A TERP structure is a
\ll twistor structure\rr (holomorphic bundle) with extra structure given by a certain extension, reality condition, and pairing.   A pure polarized TERP structure is a generalization of a pure polarized Hodge structure;  it is closely related to the concept of noncommutative Hodge structure defined in  \cite{KaKoPa07}.  Variations of TERP structures are solutions of the tt* equations.  The fundamental monodromy data here consists of a Stokes matrix $S$ and a connection matrix $E$.  It is a nontrivial and important problem to relate properties of the solution of the nonlinear p.d.e.\ with properties of $S$ and $E$, and with the underlying geometrical object (when such exists).  Because we are able to compute $S,E$ explicitly, our solutions of the tt*-Toda equations provide concrete examples where this can be done.  

More details are given in Theorem \ref{linalg} and Remark \ref{remarks}, but let us summarize briefly here the relevant properties of our solutions.  First of all we are restricting attention to solutions which have a certain cyclic symmetry - this is the Toda property.  
As we shall show that the globally smooth solutions must all have a certain distinguished connection matrix $E$,  let us fix that.
Now consider a (local) solution with this connection matrix but arbitrary Stokes matrix $S$.  When 
 
(*) the symmetric matrix $S^{-1} + S^{-t}$ is positive definite, 
 
\no by the Riemann-Hilbert approach we prove that the solution is globally smooth.  On the other hand, it follows from our p.d.e.\ approach in \cite{GuLiXX},\cite{GuItLiXX}  that a necessary and sufficient condition for the solution to be globally smooth is the weaker condition

(**) all eigenvalues of the monodromy matrix $SS^{-t}$ have unit length.

\no This is of interest as conditions (*) and (**) play a prominent role in results and conjectures concerning the smoothness of solutions associated to variations of more general TERP structures.

Finally, we mention that Mochizuki (\cite{MoXX}) has given a complete treatment of the tt*-Toda equations for general $n$ from the viewpoint of the Hitchin-Kobayashi correspondence.  This is a very special case of his deep results generalizing the theory of Corlette-Donaldson-Hitchin-Simpson to the \ll wild\rr case of irregular singularities.  

This article is organized as follows.
In section \ref{data} we review the monodromy data.  The data at $\ze=\infty$ was computed in \cite{GuItLiXX}, but we need also the
data at $\ze=0$ (which is similar) and the connection matrix (which is more difficult; we shall obtain it only after solving the Riemann-Hilbert problem).  In section \ref{inverse} we set up the Riemann-Hilbert problem.  In section \ref{asymptotic} we show that it is solvable for $x$ near infinity and give asymptotic expressions for these solutions (this argument has been sketched already by Dubrovin in \cite{Du93}).  In section \ref{vanishing}
we give conditions for solvability of the Riemann-Hilbert problem for all $x\in(0,\infty)$, and discuss properties (*), (**) above.
We give detailed calculations only in \ll case 4a\rr  (the first of the 10 cases with $n=3,4,5$ which are listed in Table 1 of \cite{GuItLiXX}).  The results for the other cases are very similar --- see Appendix B.  In fact our method works equally well for general $n$, as we shall show in \cite{GuItLi3XX}.   In section \ref{tracywidom} we discuss the relation between our results and the articles \cite{TrWi98}, \cite{Wi97} of Tracy and Widom. 

Acknowledgements: 
The first author was partially supported by JSPS grant (A) 21244004 and by Waseda University grants 2013A-868 and 2013B-083, and the second author was partially supported by NSF grant  DMS-1001777.  Both are grateful to Taida Institute for Mathematical Sciences for financial support and hospitality during their visits in 2012, when much of this work was done.   The first author is grateful to Claus Hertling for many detailed comments and for advice about the project.

\section{The monodromy data}\label{data}

We begin by giving the Stokes data for (\ref{ode}) at $\ze=\infty$ from \cite{GuItLiXX}, together with the Stokes data at $\ze=0$ (which is similar). We also define the connection matrix, which will be computed later (Theorem \ref{alldata}). 
We shall give full details in this section only for the case $n=3$, $l=4$ 
(case 4a of Table 1 of \cite{GuItLiXX}), as the other cases are very similar.

\subsection{Formal solutions}\label{formalsolutions} $ $

Let $w_0,w_1,w_2,w_3$ be a solution to 
(\ref{ost}),({\ref{as}) for $x=\vert z\vert$ on some open interval.  We write
\begin{equation}
 w=\diag(w_0,w_1,w_2,w_3),\quad
e^w=\diag(e^{w_0},e^{w_1},e^{w_2},e^{w_3})
\end{equation}
and 
\begin{equation}
W=
 \left(
\begin{array}{c|c|c|c}
\vphantom{(w_0)_{(w_0)}^{(w_0)}}
 & \!e^{w_1\!-\!w_0}\! & &  
 \\
\hline
\vphantom{(w_0)_{(w_0)}^{(w_0)}}
  &  & \!e^{w_2\!-\!w_1}\!  & \\
\hline
\vphantom{(w_0)_{(w_0)}^{(w_0)}}
  & &  &  \!e^{w_3\!-\!w_2}\!\!
\\
\hline
\vphantom{(w_0)_{(w_0)}^{(w_0)}}
\!\! e^{w_0\!-\!w_3} \!\!  & &  &  \!
\end{array}
\right).
\end{equation}
Let
\[ 
\Pi=
\bp
  & \!1\! & & \\
 & & \!1\! & \\
  & & & \!1\\
1\!   & & &
   \ep.
\]
Then we have 
$W=e^{-w} \, \Pi \, e^{w}$, $W^t=e^{w} \, \Pi^t \, e^{-w}$.
Moreover,  $\Pi= \Om \, d_4\,  \Om^{-1}$,  where
\begin{equation}
\Om=
\bp
1 & 1& 1 & 1\\
1 & \om & \om^2 & \om^3 \\
1 & \om^2 & \om^4 & \om^6 \\
1 & \om^3 & \om^6 & \om^9 
\ep,
\ \  
d_4=
\bp
1 & & &\\
 & \!\om\! & & \\
  & & \!\om^2\! & \\
  & & & \!\!\om^3
\ep,
\end{equation}
and $\om = e^{{2\pi \i}/4}=\sqrt{-1}$.   Thus the leading terms at the poles 
of (\ref{ode}) can be diagonalized as
\begin{gather*}
W = P_0 \, d_4\,  P_0^{-1},  
\ \ 
P_0 = e^{-w} \Om\\
W^t = P_\infty \, d_4\,  P_\infty^{-1},  
\ \ 
P_\infty = e^w \Om^{-1}.
\end{gather*}
We note that $\det P_0=16\, \om^3, \det P_\infty=\om/16$.
It is a classical result (see Proposition 1.1 of \cite{FIKN06}) that one obtains unique formal solutions of (\ref{ode}) at
$\ze=0,\infty$ respectively of the form
\begin{gather*}
\Psiz_f = P_0\left( I + \sum_{k\ge 1} \psiz_k \ze^k \right) e^{\frac1\ze  d_4}
\\
\Psii_f = P_\infty\left( I + \sum_{k\ge 1} \psii_k \ze^{-k} \right) e^{x^2 \ze  d_4}.
\end{gather*}
It is easily verified that no logarithmic term appears here, i.e.\ the formal monodromy is trivial.

As explained in \cite{GuItLiXX},  equation (\ref{ode}) has certain
symmetries, which can be expressed as the following properties of
$f=\Psi_\ze \Psi^{-1}$:

\no{\em Cyclic symmetry: }  $d_4^{-1} f(\om\ze) d_4 = \om^{-1} f(\ze)$

\no{\em Anti-symmetry: }   $\De f(-\ze)^t \De = f(\ze)$, \quad 
$\De=\bsp
 & & & \!1\\
 & & \!1\! & \\
 & \!1\! & &\\
 1\! & & &
\esp$

\no{\em Reality: }   $\overline{  f(\bar\ze) } = f(\ze)$

This gives the following symmetries of the formal solutions:
    
\begin{lemma}\label{formalsymmetries}   \ \ \ \ \ 

\no{Cyclic symmetry: }  
\newline
(1a) $d_4^{-1} \,\Psiz_f(\om\ze)\,\Pi^{-1}=\Psiz_f(\ze)$
\newline 
(1b) $d_4^{-1} \,\Psii_f(\om\ze)\,\Pi=\Psii_f(\ze)$

\no{Anti-symmetry: }  
\newline
(2a) $\De\, \Psiz_f(-\ze)^{-t}\,\Om\De\Om=\Psiz_f(\ze)$,\quad $\Om\De\Om=4d_4^{-1}$
\newline
(2b) $\De\, \Psii_f(-\ze)^{-t}\,(\Om\De\Om)^{-1}=\Psii_f(\ze)$

\no{Reality: }  
\newline
(3a) $\overline{\Psiz_f(\bar\ze)}\ C = \Psiz_f(\ze)$, \quad
$C=\Om \,\bar\Om^{-1}
=\bsp
1\! & & & \\
 & & & \!1\\
 & & \!1\! & \\
 & \!1\! & &
\esp$
\newline
(3b) $\overline{\Psii_f(\bar\ze)}\ C = \Psii_f(\ze)$
\end{lemma}

\subsection{Stokes sectors and Stokes matrices} $ $

There is a unique solution $\Psi$ of (\ref{ode}) with asymptotic expansion $\Psi_f$ on any Stokes sector, and the Stokes matrices are the constant matrices which relate solutions on adjoining Stokes sectors.  This principle is explained in detail in section 4 of \cite{GuItLiXX}, so we give only a brief list of facts here.  

As initial Stokes sectors at $\ze=0$ and $\ze=\infty$ we choose
\begin{gather*}
\Omz_1=\{ \ze\in\C^\ast \st -\tfrac{3\pi}4<\ar\ze <\tfrac\pi2\}
\\
\Omi_1=\{ \ze\in\C^\ast \st -\tfrac\pi2<\ar\ze < \tfrac{3\pi}4\}.
\end{gather*}
We identify these  initial sectors with subsets of 
the universal covering surface 
$\tilde\C^\ast$ (rather than $\C^\ast$), so that we can define further Stokes sectors in $\tilde\C^\ast$ by
\begin{equation}\label{sectors}
\begin{cases}
\ \ \Omz_{k+\frac14}= e^{-\frac\pi4\i} \Omz_k
\\
\ \ \Omi_{k+\frac14}= e^{\frac\pi4\i} \Omi_k
\end{cases}
\end{equation}
for any $k\in\tfrac14\Z$.
Then we have unique holomorphic solutions 
$\Psiz_k,\Psii_k$ on $\Omz_k, \Omi_k$ of (\ref{ode})
such that 
$\Psiz_k\sim\Psiz_f$ as $\ze\to0$ in $\Omz_k$ and
$\Psii_k\sim\Psii_f$ as $\ze\to\infty$ in $\Omi_k$. 
For each $k$, $\Psiz_k$ and $\Psii_k$ have analytic continuations to the whole of $\tilde\C^\ast$
which we continue to denote by $\Psiz_k$ and $\Psii_k$.
They have the following symmetries, which follow from Lemma \ref{formalsymmetries}.

\begin{lemma}\label{solutionsymmetries}   \ \ \ \ \ 

\no{\em Cyclic symmetry: }  
\newline
(1a) $d_4^{-1}\Psiz_{k-\frac12}(\om\ze)\Pi^{-1}=\Psiz_k(\ze)$
\newline
(1b) $d_4^{-1}\Psii_{k+\frac12}(\om\ze)\Pi=\Psii_k(\ze)$

\no{\em Anti-symmetry: }  
\newline
(2a) $\De \Psiz_{k-1}(e^{\i\pi}\ze)^{-t}\,4d_4^{-1}=\Psiz_k(\ze)$
\newline
(2b) $\De \Psii_{k+1}(e^{\i\pi}\ze)^{-t}\tfrac14 d_4 =\Psii_k(\ze)$

\no{\em Reality: } 
\newline
(3a) $\overline{  \Psiz_{\frac74-k}(\bar\ze)  } C = \Psiz_k(\ze)$
\newline
(3b) $\overline{  \Psii_{\frac74-k}(\bar\ze)  } C = \Psii_k(\ze)$
\end{lemma}

On overlapping sectors (and hence on all of $\tilde\C^\ast$) these solutions must be related by
\begin{gather*}
\Psiz_{k+\scriptstyle\frac14} = \Psiz_k \Qz_k
\\
\Psii_{k+\scriptstyle\frac14} = \Psii_k \Qi_k
\end{gather*}
for certain constant matrices $\Qz_k,\Qi_k$ (independent of $\ze$,
and also independent of $x$ because of the isomonodromy property).

The Stokes matrices 
\[
\Sz_k=\Qz_{k}\Qz_{k+\scriptstyle\frac14}\Qz_{k+\scriptstyle\frac24}
\Qz_{k+\scriptstyle\frac34},
\ \ 
\Si_k=\Qi_{k}\Qi_{k+\scriptstyle\frac14}\Qi_{k+\scriptstyle\frac24}
\Qi_{k+\scriptstyle\frac34}
\]
satisfy
$\Psiz_{k+1}=\Psiz_k \Sz_k$, $\Psii_{k+1}=\Psii_k \Si_k$, and we have $\Sz_{k+2}=\Sz_k$, $\Si_{k+2}=\Si_k$ for all $k$.    They constitute the \ll minimal presentation\rr of the Stokes data, but their factors $\Qz_k,\Qi_k$ are more convenient for describing the symmetries.

\begin{lemma}\label{stokessymmetries}   \ \ \ \ \ 

\no{Cyclic symmetry: }  
\newline
(1a) $\Qz_{k+\scriptstyle\frac12} = \Pi\,  \Qz_k \, \Pi^{-1}$
\newline
(1b) $\Qi_{k+\scriptstyle\frac12} = \Pi\,  \Qi_k \, \Pi^{-1}$

\no{Anti-symmetry: }  
\newline
(2a) $\Qz_{k+1} = 
d_4 \, \Qz_k{}^{-t}\, d_4^{-1}$
\newline
(2b) $\Qi_{k+1} = 
d_4^{-1} \, \Qi_k{}^{-t}\, d_4$

\no{Reality: }  
\newline
(3a) $\Qz_k= 
C
\barQz_{ {\scriptstyle\frac32} - k} {}^{-1}
C
$
\newline
(3b) $\Qi_k= 
C
\barQi_{ {\scriptstyle\frac32} - k} {}^{-1}
C
$
\end{lemma}

From these symmetries and from Lemma \ref{zandi} (see below), all $\Qz_k,\Qi_k$ are determined by $\Qi_1,\Qi_{1\scriptstyle\frac14}$,  which (see \cite{GuItLiXX})  have
the form
\[
\Qi_1=
\bp
1 & & & \\
s_1 & 1 & & \\
 & & 1 & -\bar s_1 \\
  & & & 1
\ep,
\ \ 
\Qi_{1\scriptstyle\frac14}=
\bp
1 & & & \\
 & 1 & & s_2\\
 & & 1 &  \\
  & & & 1
\ep
\]
where 
$s_1=\om^{\scriptstyle\frac32} s_1^{\R}$, 
$s_2=\om^{3} s_2^{\R}$,
and $s_1^{\R},s_2^{\R}\in\R$.  A list of all $\Qi_k$ can be found in Appendix A.

\subsection{Monodromy} $ $

Each solution $\Psiz_k,\Psii_k$ has monodromy around its singular point.  It will suffice to give the monodromy matrix for $\Psii_1$.  From the triviality of the formal monodromy (section \ref{formalsolutions}), we obtain
\[
\Psii_1(e^{-2\pi\i}\ze) = \Psii_1(\ze)\Si_1\Si_2,
\]
so the monodromy matrix of  $\Psii_1$ is $\Si_1\Si_2$. By repeated application
of Lemma \ref{stokessymmetries} (1b), we have $\Si_2=\Pi^2 \Si_1 \Pi^2$ and
hence
\[
\Si_1\Si_2=(\Si_1\Pi^2)^2=
(\Qi_1 \Qi_{1\scriptstyle\frac14}\Pi)^4.
\]
The characteristic polynomial of $\Qi_1 \Qi_{1\scriptstyle\frac14}\Pi$
is $\la^4 - s_1 \la^3 - s_2 \la^2 + \bar s_1 \la - 1$.  

\subsection{Connection matrix} $ $

Although we have omitted the derivation of the formulae for the symmetries at $\ze=0$, they can be obtained immediately from the formulae at $\ze=\infty$ and the following reality condition (\cite{GuItLiXX}, section 4, Step 1).

\no{\em Loop group reality: }   
$\De 
\overline{
f 
\left(
\tfrac{1}{x^2\bar\ze  \vphantom{ T^{T^T} }  }
\right) 
}
\De = -x^2 \ze^2 f(\ze)$ 

\no Combining this with the earlier reality condition, we obtain
\[
\De 
f 
\left(
\tfrac{1}{x^2\ze  \vphantom{ T^{T^T} }  }
\right) 
\De = -x^2 \ze^2 f(\ze).
\]
This gives
$\De 
\Psii_f
\left(
\tfrac{1}{x^2\ze  \vphantom{ T^{T^T} }  }
\right) 
4d_4^{-1} = \Psiz_f(\ze)$, hence 
$\De 
\Psii_k
\left(
\tfrac{1}{x^2\ze  \vphantom{ T^{T^T} }  }
\right) 
4d_4^{-1} = \Psiz_k(\ze)$, from which
we deduce the following simple relation:

\begin{lemma}\label{zandi}
$\Qz_k=d_4 \, \Qi_k\, d_4^{-1}.$
\end{lemma}

Finally we define the connection matrices $E_k$ by
\[
\Psii_k=\Psiz_k E_k.
\]
The connection matrix $E_1$ generates all $E_k$, as it follows from $\Psii_{k+\scriptstyle\frac14} = \Psii_k \Qi_k$ that
\begin{equation}\label{connectionshift}
d_4^{-1} E_k = \Qi_{k-\frac14} {}^{-1} \, d_4^{-1} E_{k-\frac14} \, \Qi_{k-\frac14}.
\end{equation}
The symmetries of $\Psiz_k,\Psii_k$ produce the following symmetries of $E_k$:

\begin{lemma}\label{connectionsymmetries}   \ \ \ \ \ 

\no (1) \no{\em Cyclic symmetry: }  $d_4^{-1} E_k= \om \ (\Qi_k \Qi_{k+\frac14} \Pi)
\ 
d_4^{-1} E_k\  (\Qi_k \Qi_{k+\frac14} \Pi)$

\no (2) \no{\em Anti-symmetry: }   $d_4^{-1} E_k = -\tfrac1{16}(d_4^{-1}\Pi^2)\ (d_4^{-1} E_k)^{-t} \ (d_4^{-1}\Pi^2)^{-1}$  

\no (3) \no{\em Reality: }  $E_k = C \bar E_{\frac74-k} C$
\end{lemma}

\no We note that $\det E_1=-1/256$ (from $\Psii_1=\Psiz_1 E_1$
and $\det \Psiz_1=\det P_0=16\,\om^3$, $\det \Psii_1=\det P_\infty=\om/16$).

\section{The Riemann-Hilbert problem}\label{inverse}

\subsection{Motivation}\label{motivation} $ $

Let us assume that $w_0(x),\dots,w_n(x)$ is a solution of (\ref{ost}),(\ref{as}) for $x=\vert z\vert$ in some  nonempty open interval.  Then we obtain holomorphic solutions $\Psiz_k,\Psii_k$ of (\ref{ode}) on the universal covering space $\tilde\C^\ast$ of $\C^\ast$, as explained in the previous section. 

Recall that $\Psiz_k$ was originally defined on the sector $\Omz_k$ and then extended to
$\tilde\C^\ast$ by analytic continuation.   Let $\fOmz_k=\pi(\Omz_k)$, where $\pi:\tilde\C^\ast\to\C^\ast$ is the covering map.  Then we obtain a holomorphic function 
$\fPsiz_k$ on $\fOmz_k$ in the obvious way, i.e.\  $\fPsiz_k(\pi(\ze))=\Psiz_k(\ze)$.
Similarly we define $\fPsii_k$ on $\fOmi_k$.  We note that
\[
\fPsiz_{k+2}=\fPsiz_k,\quad \fPsii_{k+2}=\fPsii_k
\]
for all $k$.
These functions can be used to construct a \ll sectionally-holomorphic\rr function $\fPsi$ on $\C^\ast$, whose jumps (discontinuities) along certain contours are given in terms of the matrices  $\Qz_k,\Qi_k,E_k$ (the monodromy data).   This motivates the Riemann-Hilbert problem of finding all sectionally-holomorphic functions which have the same jumps as $\fPsi$.   

In the most favourable situation,  
$\fPsi$ is determined (up to some normalization) by these jumps, i.e.\  the Riemann-Hilbert problem has an essentially unique solution.  Then,  the entire argument can be reversed: we can deduce the existence of a solution $w_0(x),\dots,w_n(x)$ of (\ref{ost}),(\ref{as}).   
Evidently the success of the method depends on choosing carefully the contours and the jumps.  Moreover, to obtain a solvable Riemann-Hilbert problem, it will be necessary to modify the above description, and this results in jumps which depend on $x$; thus, the conclusion will also depend on $x$, i.e.\ there will be some restriction on the domain of the predicted solution $w_0,\dots,w_n$.  

So far this discussion applies to any (local) solution $w_0,\dots,w_n$.  However, we are interested in recovering the solutions from \cite{GuItLiXX} which are globally defined on $(0,\infty)$.  For these we already know the Stokes data $s_1^\R,s_2^\R$ (see section 2 of \cite{GuItLiXX}), but we did not yet compute $E_1$. We shall formulate a
Riemann-Hilbert problem with arbitrary real $s_1^\R,s_2^\R$  and a certain postulated value of $E_1$, then investigate when the solutions of this Riemann-Hilbert problem produces the solutions of \cite{GuItLiXX}.   
As well as giving a completely different alternative proof of the existence of these solutions, this approach allows us to prove that the connection matrix $E_1$ is in fact the postulated one.

In order to carry out the above plan, let us choose (after some experimentation) the sectionally-holomorphic function shown in 
Fig.\ \ref{unfolded}.   
\begin{figure}[h]
\begin{center}
\includegraphics[scale=0.6, trim= 40  170  0  120]{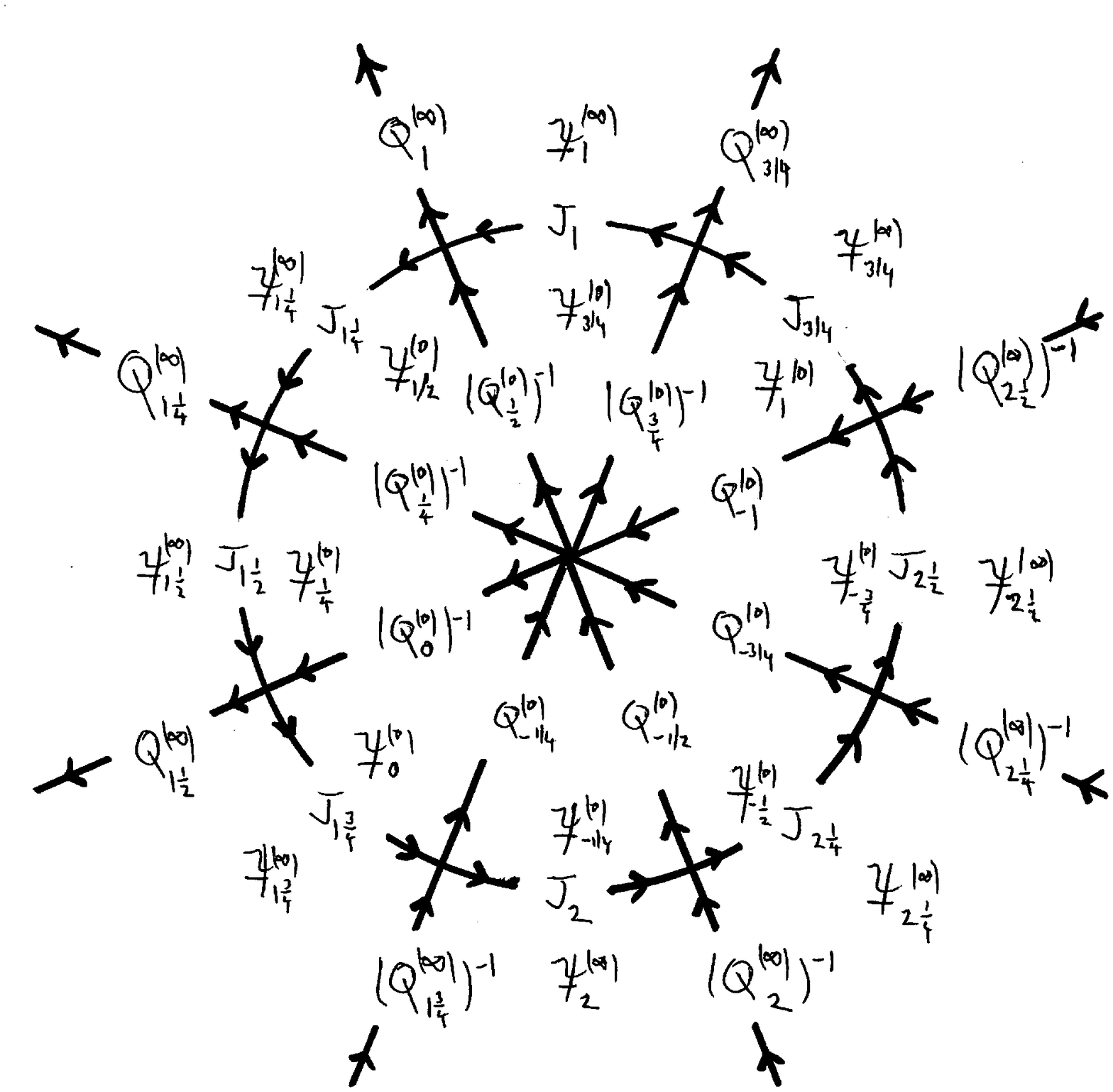}
\end{center}
\caption{The contour $\Ga_1$}\label{unfolded}
\end{figure}
The eight rays in this diagram have arguments  $\tfrac\pi8 + \tfrac\pi4\Z$.  The circle is the unit circle.   The orientations of the rays and curves are chosen arbitrarily.  We call the oriented contour $\Ga_1$.  
The jumps are defined using the convention of \cite{FIKN06}:  

\begin{definition} If 
$\fPsi_{\text{left}}$  
is defined on a region to the left of (and including) an oriented contour $\Ga$, and 
$\fPsi_{\text{right}}$ 
is defined similarly to the right, then the jump $\Xi$ is the function on $\Ga$ defined by 
$\fPsi_{\text{left}}=\fPsi_{\text{right}}\, \Xi$.   
It is assumed here that each function extends to a slightly larger open region which includes the contour.
\end{definition}
 
Away from the $\pi/8$ ray, the jumps  on the rays of the contour $\Ga_1$ in Fig.\ \ref{unfolded} follow immediately from the definition of $\Qz_k$, $\Qi_k$, as 
$\fPsii_k=\Psii_k$ and $\fPsiz_k=\Psiz_k$
here.  For $\ze$ on the outer part of the $\pi/8$ ray, we must prove that 
$\fPsi_{2\frac12}(\ze)=\fPsi_{\frac34}(\ze)(\Qi_{2\frac12})^{-1}$.   We have $\fPsii_{\frac34}(\ze)=\Psii_{\frac34}(\ze)$, but 
$\fPsii_{2\frac12}(\ze)=\Psi_{2\frac12}(e^{2\pi\i}\ze)
=\Psi_{2\frac34}(e^{2\pi\i}\ze)(\Qi_{2\frac12})^{-1} = \Psi_{\frac34}(\ze)(\Qi_{2\frac12})^{-1}$, so the jump is $(\Qi_{2\frac12})^{-1}$.  Similarly, for $\ze$ on the inner part of the $\pi/8$ ray, we have $\fPsiz_{-\frac34}(\ze)=
\Psiz_{-\frac34}(e^{2\pi\i}\ze)=\fPsiz_{-1}(e^{2\pi\i}\ze)\Qz_{-1}=
\Psiz_{1}(\ze)\Qz_{-1}=\fPsiz_{1}(\ze)\Qz_{-1}$, so the jump here is $\Qz_{-1}$, as required.  

The jumps on the segments of the circle can be expressed in terms of the connection matrices $E_k$ (hence in terms of $E_1$).  Namely, if write $\Psiz_{\frac74-k}=\Psii_k J_k$ for $k=\tfrac34,1,\dots,2\tfrac12$, then we have 
\begin{equation}\label{jumpsoncircle}
J_k=
( E_{\frac74-k} \Qi_{\frac34}\Qi_{1}\cdots\Qi_{k-\frac14} )^{-1}
\quad\text{for}\ k=1,\dots,2\tfrac12.
\end{equation}
Altogether, the jumps constitute a (piecewise-continuous) function on the contour $\Ga_1$, and we denote this function by $\Xi_1$.

\subsection{Riemann-Hilbert problem (provisional version)}\label{rh1}   $ $

Motivated by the discussion above (and by Proposition \ref{jumponcircle} below), let us pose a provisional Riemann-Hilbert problem as follows:

\no{\bf Riemann-Hilbert problem (1):}   {\em 
Let $s_1^\R,s_2^\R\in\R$. Define
matrices $\Qi_k$ as in Appendix A,
and define $\Qz_k=d_4 \Qi_k d_4^{-1}$ (as in Lemma \ref{zandi}). 
Let 
\[
E_1 = \tfrac14C \Qi_{\frac34},
\quad C=
\bsp
1\! & & & \\
 & & & \!1\\
 & & \!1\! & \\
 & \!1\! & &
\esp.
\]
Define matrices $E_k$ by formula (\ref{connectionshift}),
and define matrices $J_k$ by formula (\ref{jumpsoncircle}).
Given these matrices $\Qz_k,\Qi_k,J_k$ (which constitute $\Xi_1$),  the problem is to find a sectionally-holomorphic function (preferably unique) whose jumps on the contour $\Ga_1$ are given by the piecewise continuous function $\Xi_1$, and which have the same essential singularities at the points $\ze=0$ and $\ze=\infty$  as the formal solutions 
$\Psiz_f,\Psii_f$ respectively.
}

From the explicit form of $ \Qi_{\frac34}$, we have
\begin{equation}\label{qthreequarters}
d_4 \,\Qi_{\frac34}\, d_4^{-1} =   \Qi_{\frac34}{}^{-1} = \barQi_{\frac34}
\ \text{and}\
C \Qi_{\frac34} = \Qi_{\frac34} C.
\end{equation}
Using this, it may be verified 
that $E_k$ does in fact satisfy all the formulae of Lemma \ref{connectionsymmetries}, so it is at least a \ll valid candidate\rrr, even though we have not yet proved that it occurs as part of the monodromy data of a solution of the tt*-Toda equations.  Furthermore:

\begin{proposition}\label{jumponcircle} We have $J_k=4C$ for $k=\tfrac34,1,\dots,2\tfrac12$.
That is, all the jumps on the circle in Fig.\ \ref{unfolded} are equal to $4C$.
\end{proposition} 

\begin{proof} We shall prove by induction that 
$\Psii_k=\Psiz_{\frac74-k}\, \tfrac14C$.  This implies
$\fPsii_k=\fPsiz_{\frac74-k}\, \tfrac14C$, i.e.\ $J_k=(\tfrac14 C)^{-1}=4C$. 

From Lemma \ref{solutionsymmetries} we have  $d_4^{-1}\Psii_{k+\frac12}(\om\ze)\Pi=\Psii_k(\ze)$ (by part (b)) 
$=\Psiz_{\frac74-k}(\ze)\, \tfrac14C$ (induction hypothesis) 
$=d_4^{-1}\Psiz_{\frac54-k}(\om\ze)\Pi^{-1}\, \tfrac14C$ (by part (a)).
Hence 
$\Psii_{k+\frac12}(\om\ze)=
\Psiz_{\frac54-k}(\om\ze)\Pi^{-1}\, \tfrac14C\Pi^{-1}
=\Psiz_{\frac54-k}(\om\ze)\, \tfrac14C$ (by formula (F3), Appendix A). This is the result for $k+\frac12$.

To start the induction we check the cases $k=\frac34,1$. For $k=\frac34$ we have $\Psii_{\frac34} \Qi_{\frac34} = \Psii_1 = \Psiz_1 E_1=
\Psiz_1\,\frac14 C  \Qi_{\frac34}$, so $\Psii_{\frac34}= \Psiz_1\,\frac14C$.
For $k=1$ we have $\Psii_1=\Psii_{\frac34} \Qi_{\frac34} =
\Psiz_{\frac34} E_{\frac34} \Qi_{\frac34}$ and this is
$\Psiz_{\frac34} d_4 \Qi_{\frac34} d_4^{-1} E_1$ by formula (\ref{connectionshift}).  Formula (\ref{qthreequarters}) gives 
$d_4 \,\Qi_{\frac34}\, d_4^{-1} =   (\Qi_{\frac34})^{-1}$, so we have
$\Psii_1=\Psiz_{\frac34}  (\Qi_{\frac34})^{-1} \Qi_{\frac34} \, \tfrac14C
=\Psiz_{\frac34}  \, \tfrac14C$, as required.
\end{proof} 

This allows us to simplify
the Riemann-Hilbert problem.  Namely, if we replace 
$\fPsiz_k$ by $\fPsiz_k 4C$ in Fig.\ \ref{unfolded}, then all jumps on the circle become the identity matrix --- there is in fact no discontinuity across the circle.  This suggests a new Riemann-Hilbert problem based on Fig.\ \ref{simplified}. The contour $\Ga_2$ is obtained from the contour $\Ga_1$ of Fig.\ \ref{unfolded}
by removing the circle. The piecewise continuous function giving the indicated jumps on the contour $\Ga_2$ will be denoted $\Xi_2$. 
\begin{figure}[h]
\begin{center}
\includegraphics[scale=0.6, trim= 40  160  0  120]{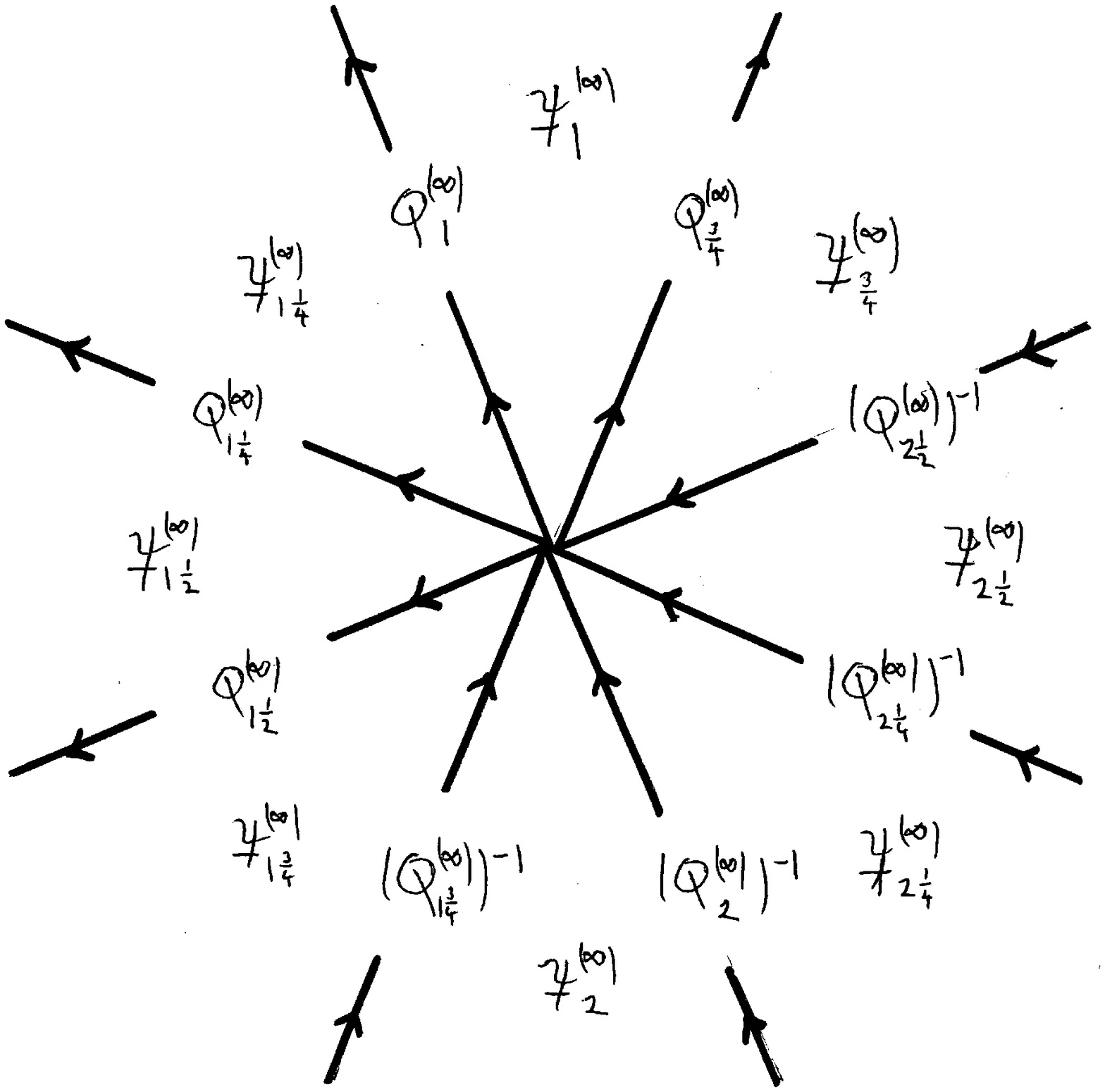}
\end{center}
\caption{The simplified contour $\Ga_2$}\label{simplified}
\end{figure}

\subsection{The Riemann-Hilbert problem}\label{rh2} $ $

From now on we shall reformulate the Riemann-Hilbert problem in the manner of Chapter 3 
of \cite{FIKN06}, in order to use the criteria for solvability given there.  That is, given 

(i) an oriented contour $\Ga$ (possibly with nodes, i.e.\ points of self-intersection), and

(ii) a map $G:\Ga\to \textrm{GL}_n\C$ (the space of invertible complex $n\times n$ matrices),

\no we seek a holomorphic map $Y:\C-\Ga\to \textrm{GL}_n\C$ such that
\[
Y_{\text{left}}=Y_{\text{right}} G\quad\text{and}\quad Y\to I \ \text{as}\ \ze\to\infty,
\]
where $Y_{\text{left}}$ and $Y_{\text{right}}$ are the (pointwise) limits of $Y$ from the left and right sides of $\Ga$, respectively.  We shall require that $(G,\Ga)$ satisfy the following additional conditions (pages 102/3 of \cite{FIKN06}):

(1) $\Ga$ is a finite union of smooth components $\Ga^{(i)}$, and each $G^{(i)}=G\vert_{\Ga^{(i)}}$ admits an analytic continuation in a neighbourhood of $\Ga^{(i)}$,

(2)  $G(\ze)$ approaches $I$ exponentially as $\ze\to\infty$ along any infinite component of $\Ga$,

(3) $\det G(\ze)=1$, and

(4) at any node of $\Ga$,  for which the intersecting components are 
$\Ga^{(i_1)},\dots,\Ga^{(i_N)}$ (listed in anticlockwise order around the node),  we have
$(G^{(i_1)})^{\eps_1}\dots(G^{(i_N)})^{\eps_N}=I$, where $\eps_i$ is $+1$ or $-1$ according to whether $\Ga^{(i)}$ is oriented outwards or inwards at that point.

As in the case of the provisional problem above, we shall investigate first the expected properties of the solution, then formulate a Riemann-Hilbert problem based on those properties.  Let us consider the situation of Fig.\ \ref{simplified}.
As $\ze\to\infty$, from the definition of $\Psii_k$, we have 
$P_\infty^{-1}\Psii_k e^{-x^2\ze d_4}\to I$.  On the other hand, 
as $\ze\to 0$, by using Lemma \ref{jumponcircle}, we obtain
\[
\Psii_k=\Psiz_{\frac74-k}\ \tfrac14C\sim 
P_0 (I + O(\ze)) e^{\frac1\ze d_4} \tfrac14C
= \tfrac14P_0 C (I + O(\ze)) e^{\frac1\ze d_4^{-1}}.
\]
Here we have used $d_4 C = C d_4^{-1}$ (formula (F4), Appendix A).

This suggests that we introduce
\[
Y_k(\ze,x)=P_\infty^{-1}\ \fPsii_k (\ze,x)\,  e^{-x^2 \ze  d_4 -\frac1\ze  d_4^{-1}}
\]
as an appropriate modification of $\fPsi_k$. Then the jumps
$G_2$ on the contour $\Ga_2$ are given by $G_2=e\, \Xi_2\, e^{-1}$, where $e(\ze,x)=e^{x^2 \ze  d_4 +\frac1\ze  d_4^{-1}}$.  We note that $Y_k$ is defined on the sector $\fOmi_k$, but for the purposes of Fig.\ \ref{simplified} we shall consider its restriction to a smaller sector of angle $\pi/4$. We have $Y_{k+2}=Y_k$ for all $k$. 

The following problem will be our main focus.

\no{\bf Riemann-Hilbert problem (2):}  {\em 
Let $s_1^\R,s_2^\R\in\R$.  Define
matrices $\Qi_k$ as in the previous section.
For these matrices,  find a sectionally-holomorphic function $Y=\{Y_k\}$,
such that $Y\to I$ as $\ze\to\infty$, 
whose jumps on the contour $\Ga_2$ are given by $G_2=e\,\Xi_2\, e^{-1}$,
where 
$e(\ze,x)=e^{x^2 \ze  d_4 +\frac1\ze  d_4^{-1}}$ and
$\Xi_2$ is as shown in Fig.\ \ref{simplified}.   
}

By construction, conditions (1)-(4) are satisfied.  We shall also need $G_2\to I$ as $\ze\to 0$, so let us verify this next.
Let $\ze=k e^{\i\th}$, $k>0$, be any ray of the contour $\Ga_2$.  On this ray the function $\Xi_2$ is a constant matrix, and the $(i,j)$ entry of
$G_2=e \,\Xi_2\, e^{-1}$ is
\[
(G_2(\ze,x))_{ij}=
\begin{cases}
(\Xi_2)_{ij}\ e^{
\left( 
\frac{2}k e^{\i \boxe}  \ +\    2 k x^2 e^{-\i \boxe}
\right)
}
\ \ \text{if $i-j$ is even}
\\
(\Xi_2)_{ij}\ e^{
\left( 
\frac{\sqrt 2}k e^{\i \boxe}  \ +\    \sqrt 2 k x^2 e^{-\i \boxe}
\right)
}
\ \ \text{if $i-j$ is odd}
\end{cases}
\]
where the boxed angle (which depends on $(i,j)$) is indicated in the diagram below:
\[
\begin{array}{|c|c|c|c|}
\hline
\vphantom{\dfrac12}
0 & -\th+\tfrac\pi4 &  -\th+0&   -\th-\tfrac\pi4
\\
\hline
\vphantom{\dfrac12}
 -\th+\tfrac{5\pi}4  & 0 &  -\th-\tfrac{\pi}4 &  -\th-\tfrac{\pi}2
  \\
\hline
\vphantom{\dfrac12}
 -\th+\pi &  -\th+\tfrac{3\pi}4 & 0 &    -\th-\tfrac{3\pi}4
\\
\hline
\vphantom{\dfrac12}
 -\th+\tfrac{3\pi}4 & -\th+\tfrac{\pi}2  &  -\th+\tfrac{\pi}4&  0
\\
\hline
\end{array}
\]
From the list of matrices $\Qi_k$ in Appendix A, we see that different rays contribute to different entries of $\Xi_2$.  In the following two diagrams we list 
these contributions
and the angle $\th$  of the corresponding ray, respectively:
\[
\begin{array}{|c|c|c|c|}
\hline
\vphantom{\dfrac{A}{A_A}}
\hphantom{\!\Qi_2{}^{-1}\!\!} & -\om^{\frac12} s^\R_1  &  -\om^3 s^\R_2&  \om^{\frac32} s^\R_1
\\
\hline
\vphantom{\dfrac{A}{A_A}}
 \om^{\frac32} s^\R_1&  &  \om^{\frac12} s^\R_1 & \om^3 s^\R_2
  \\
\hline
\vphantom{\dfrac{A}{A_A}}
 \om^3 s^\R_2 &  -\om^{\frac32} s^\R_1 &  &   \om^{\frac12} s^\R_1
\\
\hline
\vphantom{\dfrac{A}{A_A}}
 -\om^{\frac12} s^\R_1&  -\om^3 s^\R_2 &  -\om^{\frac32} s^\R_1 &  
\\
\hline
\end{array}
\ \ \ 
\begin{array}{|c|c|c|c|}
\hline
\vphantom{\dfrac{A}{A_A}}
\hphantom{\!\Qi_2{}^{-1}\!\!}& -\frac{3\pi}8 & -\frac{5\pi}8&  -\frac{7\pi}8
\\
\hline
\vphantom{\dfrac{A}{A_A}}
\frac{5\pi}8  & \hphantom{\!\Qi_2{}^{-1}\!\!} &-\frac{7\pi}8  & -\frac{9\pi}8
  \\
\hline
\vphantom{\dfrac{A}{A_A}}
\frac{3\pi}8 &  \frac{\pi}8 & \hphantom{\!\Qi_2{}^{-1}\!\!} &  -\frac{11\pi}8
\\
\hline
\vphantom{\dfrac{A}{A_A}}
\frac{\pi}8 & -\frac{\pi}8 & -\frac{3\pi}8&  \hphantom{\!\Qi_2{}^{-1}\!\!}
\\
\hline
\end{array}
\]
For example, $\Qi_{1\frac14}$ is the jump at the $7\pi/8$ (=$-9\pi/8$) ray, 
and it contributes the $(2,4)$ entry $\om^3 s_2^\R$.

When the value of $\th$ in the third diagram is substituted into the first diagram,
the resulting (off-diagonal) angle is $5\pi/8$ {\em in all cases.}
Hence we obtain the following explicit expression for the jump function $G_2$:
\begin{equation}\label{Gtwo}
G_2(\ze,x)=
\left(
\begin{array}{c|c|c|c}
\vphantom{\dfrac{A}{A_A}}
\hphantom{\!\Qi_2{}^{-1}\!\!} 1 & -\om^{\frac12} s^\R_1 \epr  &  -\om^3 s^\R_2 \eprr &  \om^{\frac32} s^\R_1  \epr
\\
\hline
\vphantom{\dfrac{A}{A_A}}
 \om^{\frac32} s^\R_1 \epr & 1 &  \om^{\frac12} s^\R_1 \epr & \om^3 s^\R_2 \eprr
  \\
\hline
\vphantom{\dfrac{A}{A_A}}
 \om^3 s^\R_2 \eprr &  -\om^{\frac32} s^\R_1 \epr &  1 &   \om^{\frac12} s^\R_1 \epr
\\
\hline
\vphantom{\dfrac{A}{A_A}}
 -\om^{\frac12} s^\R_1 \epr &  -\om^3 s^\R_2 \eprr &  -\om^{\frac32} s^\R_1 \epr &  1
\\
\end{array}
\right)
\end{equation}
where 
\[
\epr=e^{ \sqrt 2\left(\tfrac{d}{k} + x^2 \tfrac{k}{d}  \right) },  \
\eprr=e^{  2\left(\tfrac{d}{k} + x^2 \tfrac{k}{d}  \right) }, \ 
d=e^{5\pi\i/8} = \cos \tfrac{5\pi}8 + \i \sin \tfrac{5\pi}8.
\]
Since $\cos 5\pi/8<0$, we have $\lim_{k\to 0} \epr = \lim_{k\to 0} \eprr = 0$,
and hence
 $\lim_{k\to 0} G_2=I$, as required.  

Evidently there is some flexibility in the choice of contours $\Ga_1,\Ga_2$ here.  
The domain of 
$\fPsii_k$ in Fig.\ \ref{unfolded} is 
$( (8k-5)\tfrac\pi8,(8k-3)\tfrac\pi8 )$, and the original domain of definition of $\Psii_k$ is
$( (8k-12)\tfrac\pi8,(8k-2)\tfrac\pi8 )$, so the jumps are unchanged if we rotate the
contours $\Ga_1,\Ga_2$ by any angle $\al$ with $-7\pi/8<\al<\pi/8$.  
In the preceding calculation, such a rotation would also result in a negative cosine, so
the problem is still well-posed.
\begin{figure}[h]
\begin{center}
\includegraphics[scale=0.6, trim= 40  160  0  120]{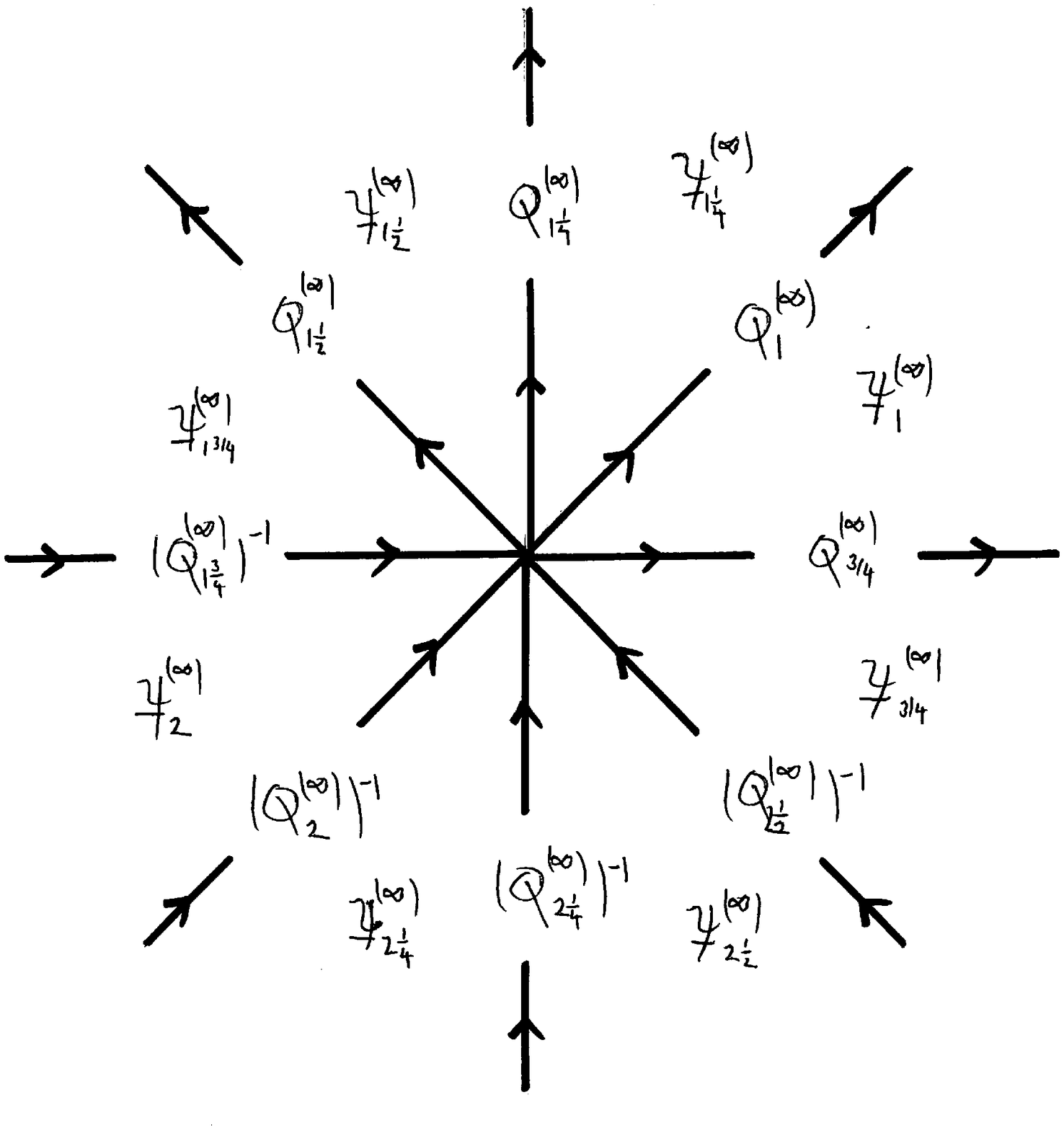}
\end{center}
\caption{The simplified and rotated contour $\Ga^\prime_2$}\label{simplifiedrotated}
\end{figure}

For example, if we take $\al=-3\pi/8$, i.e.\ replace $\th$ by $\th-3\pi/8$, then all the angles in the first diagram become $\pi$.  Then we have $c=e^{\i\pi}=-1$, and $\epr,\eprr$
become real-valued.   This will be convenient for future calculations, so let us do this.  We denote the rotated contour by $\Ga^\prime_2$ (Fig.\ \ref{simplifiedrotated}), and state the corresponding
Riemann-Hilbert problem:

\no{\bf Riemann-Hilbert problem (2)${\,}^\prime$:}  {\em 
Let $s_1^\R,s_2^\R\in\R$.  Define
matrices $\Qi_k$ as in the previous section.
For these matrices,  find a sectionally-holomorphic function $Y=\{Y_k\}$,
such that $Y\to I$ as $\ze\to\infty$, 
whose jumps on the contour $\Ga^\prime_2$ are given by $G^\prime_2=e\, \Xi^\prime_2\, e^{-1}$,
where 
$e(\ze,x)=e^{x^2 \ze  d_4 +\frac1\ze  d_4^{-1}}$ and
$\Xi^\prime_2$ is as shown in Fig.\ \ref{simplifiedrotated}. 
}

Let us spell this out in more detail.  We seek holomorphic (invertible matrix-valued) functions 
$\{Y_k\st k\in\tfrac14\Z \}$ with $Y_{k+2}=Y_k$ such that

(1) $\lim_{\ze\to\infty}Y_k=I$

(2) the domain of $Y_k$ is an open subset of $\C^\ast$ containing the sector
$(k-1)\pi \le \arg\ze \le (k-\tfrac34)\pi$

(3) on some open subset of $\C^\ast$ containing the ray $\arg\ze=(k-1)\pi$ we have $Y_k(\ze)=Y_{k-\frac14}(\ze) \ e(\ze) \Qi_{k-\frac14} e(\ze)^{-1}$

\no(we abbreviate
$Y_k(\ze,x)$, $e(\ze,x)$ to $Y_k(\ze)$, $e(\ze)$ here and in section \ref{relation}).

Because $G\to I$ as $\ze\to 0$,  the limit $\lim_{\ze\to 0}Y_k$ must exist, independently of $k$. Let us denote its value by $Y(0)$.

\subsection{Relation between the p.d.e.\ and the Riemann-Hilbert problem}\label{relation} $ $

Before attempting to solve the Riemann-Hilbert problem, we have to verify that this would in fact produce a solution of the original equations (\ref{ost}),(\ref{as}).  We shall approach this by considering the symmetries of the Riemann-Hilbert problem.

\begin{proposition} Assume that $Y$ is a solution of the Riemann-Hilbert problem with contour $\Ga^\prime_2$ and jump function $G^\prime_2$, with $Y\vert_{\ze=\infty}=I$. Then $Y$ is unique, and it has the following symmetries:

\no{\em Cyclic symmetry: }  $\Pi^{-1}\ Y_{k+\scriptstyle\frac12}(\om\ze)\  \Pi = Y_k(\ze)$

\no{\em Anti-symmetry: }   $d_4^{-1}\  Y_{k+1}(-\ze)^{-t}\  d_4= Y_k(\ze)$

\no{\em Reality: }   $C \ \overline{  Y_{\scriptstyle\frac74 -k}(\bar\ze) }\ C = Y_k(\ze)$
\end{proposition}

It is easily verified that $Y$ would inherit these properties from $\fPsi$, if 
$Y$ was obtained from $\fPsi$ as in section \ref{rh1}.  To prove the proposition we have to show that they follow from the properties of $G^\prime_2$ alone.

\begin{proof} Uniqueness follows immediately from holomorphicity and the normalization
$Y\vert_{\ze=\infty}=I$.  
We begin with the cyclic symmetry. Consider the $(k-1)\pi$ ray. Here we have 
$Y_k(\ze)=Y_{k-\frac14}(\ze) e(\ze) \Qi_{k-\frac14} e(\ze)^{-1}$.  We claim that
$\Pi^{-1}Y_{k+\scriptstyle\frac12}(\om\ze) \Pi=
\Pi^{-1}Y_{k+\scriptstyle\frac14}(\om\ze) \Pi
\  e(\ze) \Qi_{k-\frac14} e(\ze)^{-1}$ (when $\arg\ze=(k-1)\pi$). 
 Then $\{Y_k\st k\in\tfrac14\Z \}$ and $\{ \Pi^{-1}Y_{k+\scriptstyle\frac12}(\om\ze)\Pi \st k\in\tfrac14\Z \}$ solve the Riemann-Hilbert problem, and they both take the value $I$ at $\ze=\infty$, hence they are equal.
 
 To prove the claim, we compare $Y_{k+\scriptstyle\frac12}(\om\ze)$, $Y_{k+\scriptstyle\frac14}(\om\ze)$ for $\arg\ze=(k-1)\pi$.   By assumption, we have 
 $Y_{k+\scriptstyle\frac12}(\om\ze)=Y_{k+\scriptstyle\frac14}(\om\ze) e(\om\ze) \Qi_{k+\scriptstyle\frac14}  e(\om\ze)^{-1}$.  Thus, we have to show that
 $\Pi^{-1} e(\om\ze) \Qi_{k+\scriptstyle\frac14}  e(\om\ze)^{-1} \Pi = 
  e(\ze) \Qi_{k-\frac14} e(\ze)^{-1}$, i.e.\ 
  $\Qi_{k+\scriptstyle\frac14}=e(\om\ze)^{-1} \Pi e(\ze) 
  \Qi_{k-\scriptstyle\frac14} \left( e(\om\ze)^{-1} \Pi e(\ze)  \right)^{-1}$.
 As $\Pi^{-1} d_4 \Pi = -\om d_4$ (formula (F7), Appendix A) we obtain
 $e(\om\ze)^{-1} \Pi = e^{-x^2 \om \ze d_4 - \frac1{\om\ze} d_4^{-1}   } \Pi = 
 \Pi e^{  -x^2  \ze d_4 - \frac1{\ze} d_4^{-1}  } = \Pi e(\ze)$.  Thus
 $e(\om\ze)^{-1} \Pi e(\ze) =\Pi$, and we have to show that
 $\Qi_{k+\scriptstyle\frac14}=\Pi \Qi_{k-\scriptstyle\frac14} \Pi^{-1}$.  But this is 
 formula (1b) of Lemma \ref{stokessymmetries},  the cyclic symmetry of $\Qi_k$.

To verify the anti-symmetry property, we 
need
$d_4^{-1} Y_{k+1}(-\ze)^{-t} d_4= 
d_4^{-1} Y_{k+\scriptstyle\frac34}(-\ze)^{-t} d_4 \ e(\ze) \Qi_{k-\frac14} e(\ze)^{-1}$
on the ray $\arg\ze =(k-1)\pi$. By assumption, we have 
$Y_{k+1}(-\ze)=Y_{k+\scriptstyle\frac34}(-\ze) e(-\ze) \Qi_{k+\scriptstyle\frac34}  e(-\ze)^{-1}$.  Thus, we have to show that 
$d_4^{-1}\left(  e(-\ze) \Qi_{k+\scriptstyle\frac34}  e(-\ze)^{-1} \right)^{-t} d_4
=e(\ze) \Qi_{k-\frac14} e(\ze)^{-1}$.  
Since $d_4$ commutes with $e$, and $e(-\ze)=e(\ze)^{-1}$, we have to show
$d_4^{-1} (\Qi_{k+\scriptstyle\frac34})^{-t} d_4 = \Qi_{k-\scriptstyle\frac14}$.
 But this is 
 formula (2b) of Lemma \ref{stokessymmetries},  the anti-symmetry property of $\Qi_k$.
 
The reality property can be established in the same way.  We have to show that 
$C  \overline{  Y_{\frac74 -k}(\bar\ze) } C = 
C  \overline{  Y_{\frac84 -k}(\bar\ze) } C  \ 
e(\ze) \Qi_{k-\frac14} e(\ze)^{-1}$ on the ray $\arg\ze =(k-1)\pi$. By assumption, we have 
$  Y_{\frac84 -k}(\bar\ze) =
Y_{\frac74 -k}(\bar\ze) \ 
 e(\bar\ze) \Qi_{\scriptstyle\frac74 -k} e(\bar\ze)^{-1}  $.
 Thus, we have to show that 
 $C  \overline{ e(\bar\ze)} \barQi_{\frac74 -k} {}^{-1} \overline{ e(\bar\ze)}^{-1}  C = e(\ze) \Qi_{k-\frac14} e(\ze)^{-1}$, i.e.\
$\Qi_{k-\frac14} = e(\ze)^{-1} C \overline{ e(\bar\ze)}\ 
\barQi_{\frac74 -k} {}^{-1} \left(  e(\ze)^{-1} C \overline{ e(\bar\ze)} \right)^{-1}$. As $Cd_4 C = d_4^{-1}$ (formula (F4), Appendix A) we obtain
$e(\ze)^{-1} C \overline{ e(\bar\ze)} = C$, so we have to show that
$\Qi_{k-\frac14} = C
\barQi_{\frac74 -k} {}^{-1} C$.
But this is  formula (3b) of Lemma \ref{stokessymmetries},  the reality property of $\Qi_k$.
\end{proof}
 
\begin{corollary}  If $Y$ is as in the proposition, then the matrix $Y(0)=Y\vert_{\ze=0}$  has the following symmetries:

\no{\em Cyclic symmetry: }  $\Pi^{-1}Y(0)\Pi = Y(0)$

\no{\em Anti-symmetry: }   $d_4^{-1} Y(0)^{-t}  d_4= Y(0)$

\no{\em Reality: }   $C \bar Y(0) C = Y(0)$

\no It follows that 
$
Y(0)=
\Om
\diag(a,b,b^{-1},a^{-1})
\Om^{-1}
$
for some nonzero real numbers $a,b$. 
\end{corollary}
 
\begin{proof}  The symmetries are immediate from the proposition.  Let us consider
$M=\Om^{-1} Y(0) \Om$.
Since $Y(0)$ commutes with $\Pi$ (cyclic symmetry), $M$ commutes
with $\Om^{-1}\Pi^{-1}\Om$.  But $\Om^{-1}\Pi^{-1}\Om=d_4^{-1}$ (formula (F2) of Appendix A).  Hence $M$ is diagonal, say $\diag(a,b,c,d)$.  The antisymmetry condition 
gives $M^{-t}= (\Om d_4 \Om) M (\Om d_4 \Om)^{-1}$.  Since $\Om d_4 \Om=4\De$ 
(formula (F5)), we have  $M^{-t}= \De M \De$, i.e.\ $c=b^{-1}$, $d=a^{-1}$.
Finally the reality condition gives $\bar M=(\bar\Om^{-1} C \Om) M (\bar\Om^{-1} C \Om)^{-1}$.   Since $\bar\Om^{-1}=\tfrac14\Om$ (formula (F1)),
$\bar\Om^{-1} C \Om=\tfrac14\Om C \Om= \tfrac14\Om C^{-1} \Om$.  This
is equal to $\Om^{-1}\Om$ i.e.\ $I$ by formula (F6).  Thus $\bar M=M$, and $a,b$ must be real.
\end{proof}

\begin{proposition}\label{abpositive}
Assume that $Y$ is a solution of the Riemann-Hilbert problem with contour $\Ga^\prime_2$ and jump function $G^\prime_2$, with $Y\vert_{\ze=\infty}=I$, as above.   Assume\footnote{If $a$ or $b$ is negative, the same proof shows that $w_0,w_1,w_2,w_3$ satisfy 
the equation $(xw_{x})_x=2x[W^t,W]$.   But in the negative sign case, $w_i$ takes values in $\tfrac12\i\pi + \i\pi\Z$ rather than $\i\pi\Z$.  If both $a$ and $b$ are negative, this has no effect on $(xw_{x})_x=2x[W^t,W]$, so $w_0,w_1,w_2,w_3$ still satisfy 
(\ref{ost}),(\ref{as}),  but with the asymptotics modified in the obvious way.}
 that $a,b>0$.  Define real functions $w_0,w_1,w_2,w_3$ (modulo $\i\pi\Z$) by $a= e^{-2w_0}$, $b= e^{-2w_1}$,
and $w_2=-w_1$, $w_3=-w_0$.  Then $w_0,w_1,w_2,w_3$ satisfy (\ref{ost}),(\ref{as}).
\end{proposition}

\begin{proof}  Let us define $P_0=e^{-w}\Om$, $P_\infty=e^{w}\Om^{-1}$ where 
$\Om$ is as before and
$e^{w}=\diag(e^{w_0},e^{w_1},e^{w_2},e^{w_3})$.   Let us
introduce a new function
\[
\Psi=P_\infty Y  e^{x^2 \ze  d_4 +\frac1\ze  d_4^{-1}}
\]
Like $Y$, $\Psi$ is sectionally-holomorphic, but $\Psi_\mu\Psi^{-1}$ and $\Psi_x\Psi^{-1}$
are holomorphic for all $\mu=x\ze\in\C^\ast$.  We claim that
\[
\Psi_\mu\Psi^{-1}=\left( -\tfrac{1}{\mu^2} xW - \tfrac{1}{\mu} xw_x + xW^t\right),\quad
\Psi_x\Psi^{-1}=\left(  \tfrac1\mu W + \mu W^t\right)
\]
i.e.\ $\Psi$ satisfies the system (\ref{zcc}) of section \ref{intro}.  It follows from this that 
$w_0,w_1,w_2,w_3$ satisfy (\ref{ost}) (and they satisfy (\ref{as}) by construction).  

To prove the claim, we shall use 
\[
\Psi(\mu,x)\sim
\begin{cases}
P_\infty (I +\tfrac1\mu A_1 + \tfrac1{\mu^2} A_2 + \cdots)
e^{x\mu  d_4}
\quad\text{as}\ \mu\to\infty
\\
P_0 (I +\mu B_1 + \mu^2 B_2 + \cdots)
e^{x\frac 1\mu  d_4^{-1}} \tfrac14 C
\quad\text{as}\ \mu\to 0.
\end{cases}
\]
We obtain
\begin{align*}
\Psi_\mu\Psi^{-1}&= 
xP_\infty d_4 P_\infty^{-1} + \tfrac1\mu x P_\infty[A_1,d_4] P_\infty^{-1} + O(\tfrac1{\mu^2})
\\
&=xW^t + \tfrac1\mu x [P_\infty A_1 P_\infty^{-1},W^t] + O(\tfrac1{\mu^2})
\end{align*}
near $\mu=\infty$, and
\begin{align*}
\Psi_\mu\Psi^{-1}&= 
-\tfrac1{\mu^2} xP_0 d_4 P_0^{-1} - \tfrac1\mu x P_0[B_1,d_4] P_0^{-1} + O(1)
\\
&=-\tfrac1{\mu^2} xW - \tfrac1\mu x [P_0 B_1 P_0^{-1},W] + O(1)
\end{align*}
near $\mu=0$.  As $\Psi_\mu\Psi^{-1}$ is holomorphic on $\C^\ast$, we must have
\[
\Psi_\mu\Psi^{-1}= -\tfrac1{\mu^2} xW + \tfrac1\mu x U + xW^t
\]
where
\begin{equation}\label{u}
U=[P_\infty A_1 P_\infty^{-1},W^t] =-[P_0 B_1 P_0^{-1},W].
\end{equation}

A similar calculation for $\Psi_x\Psi^{-1}$ gives
\begin{align*}
\Psi_x\Psi^{-1}&= 
\mu P_\infty d_4 P_\infty^{-1} + (P_\infty)_x P_\infty^{-1}+
P_\infty[A_1,d_4] P_\infty^{-1} + O(\tfrac1{\mu})
\\
&=\mu W^t + w_x + [P_\infty A_1 P_\infty^{-1},W^t] + O(\tfrac1{\mu})
\end{align*}
near $\mu=\infty$, and
\begin{align*}
\Psi_\mu\Psi^{-1}&= 
\tfrac1{\mu} P_0 d_4 P_0^{-1} + (P_0)_x P_0^{-1}+
P_0[B_1,d_4] P_0^{-1} + O(\mu)
\\
&=\tfrac1{\mu} W - w_x +  [P_0 B_1 P_0^{-1},W] + O(\mu)
\end{align*}
near $\mu=0$, hence
\[
\Psi_x\Psi^{-1}= \tfrac1{\mu} W + V + \mu W^t
\]
where
\begin{equation}\label{v}
V=w_x+[P_\infty A_1 P_\infty^{-1},W^t] =-w_x + [P_0 B_1 P_0^{-1},W].
\end{equation}
Combining (\ref{u}) and (\ref{v}), we see that
\[
[P_0 B_1 P_0^{-1},W]=w_x=-[P_\infty A_1 P_\infty^{-1},W^t].
\]
Thus $U=-w_x$ and $V=0$, as required.  
\end{proof}

The domain of definition of the functions 
$w_0(x),w_1(x),w_2(x),w_3(x)$ here is simply the set of $x$ for which
the Riemann-Hilbert problem can be solved.  We investigate this
in the next two sections.

\section{Solution of the Riemann-Hilbert problem for $x$ large}\label{asymptotic}

We shall show that the Riemann-Hilbert problem for the function $G^\prime_2$ on the contour 
$\Ga_2^\prime$ 
(Fig.\ \ref{simplifiedrotated}) is solvable when the parameter $x$ is sufficiently large.  To do this we shall apply Theorem 8.1 of \cite{FIKN06}, which expresses the Riemann-Hilbert problem as an integral equation and gives a criterion for solvability.

First let us calculate $G^\prime_2(\ze,x)$ explicitly.  As in the calculation of $G_2$ in the previous section, the $(i,j)$ entry of $G_2^\prime$ is
\[
( G^\prime_2(\ze,x))_{ij}=
\begin{cases}
(\Xi_2^\prime)_{ij}\ e^{\left( 
\frac{2}k e^{\i \boxep}  \ +\    2 k x^2 e^{-\i \boxep}
\right)}
\ \ \text{if $i-j$ is even}
\\
(\Xi_2^\prime)_{ij}\ e^{\left( 
\frac{\sqrt 2}k e^{\i \boxep}  \ +\  \sqrt 2 k x^2 e^{-\i \boxep}
\right)}\ \ \text{if $i-j$ is odd}
\end{cases}
\]
where $\ze=k e^{\i\th}$.
(The boxed angle is $\pi$ for the contour $\Ga^\prime_2$.)

Let us write $l=xk$.  Then we have
\[
( G^\prime_2(\ze,x))_{ij}=
\begin{cases}
(\Xi_2^\prime)_{ij}\ e^{-2x
\left( 
\frac{1}l  \ +\     l
\right)}
\ \ \text{if $i-j$ is even}
\\
(\Xi_2^\prime)_{ij}\ e^{-\sqrt 2 x
\left( 
\frac{1}l  \ +\     l
\right)}\ \ \text{if $i-j$ is odd}
\end{cases}
\]
hence
\[
\vert \left( G^\prime_2(\ze,x)-I \right)_{ij} \vert \le
\begin{cases}
A e^{
-2x
\left( 
\frac{1}l  \ +\     l
\right)
}
\le 
A e^{-4x}
\ \ \text{if $i-j$ is even}
\\
A e^{
-\sqrt2x\left( 
\frac{1}l  \ +\     l
\right)
}
\le 
A e^{-2\sqrt2x}
\ \ \text{if $i-j$ is odd}
\end{cases}
\]
where $A$ is a constant (which depends only on $s_1^\R,s_2^\R$).
This implies solvability for sufficiently large $x$, and one has:

\begin{equation}\label{integral}
\displaystyle Y(0,x) = I  + \tfrac{1}{2\pi\ii}\int_{\Ga_2^\prime}  
\frac{G^\prime_2(\ze,x) - I}{\ze} \  d\ze + O(e^{-4\sqrt2 x}).
\end{equation}

Using the version of formula (\ref{Gtwo}) for $\Xi_2^\prime$, and integrating over the whole contour, we obtain $Y(0,x)=$
\begin{equation*}
\left(
\begin{array}{cccc}
\vphantom{\dfrac12}
1 &  \!\!\!\om^{-\frac12} s_1^\R \, \phi(\sqrt2 x)\!\! & -s_2^\R \, \phi(2 x) &  \!\!\!\om^{\frac12} s_1^\R \, \phi(\sqrt2 x)\!\!  \!\!\!
\\
\vphantom{\dfrac12}
\!\!\!\om^{\frac12} s_1^\R \, \phi(\sqrt2 x)\!\!   & 1 &   \!\!\!\om^{-\frac12} s_1^\R \, \phi(\sqrt2 x) & -s_2^\R \, \phi(2 x)\!\!
  \\
\vphantom{\dfrac12}
-s_2^\R \, \phi(2 x) &  \!\!\!\om^{\frac12} s_1^\R \, \phi(\sqrt2 x)\!\!  & 1 &   \!\!\!\om^{-\frac12} s_1^\R \, \phi(\sqrt2 x)\!\!  \!\!\!
\\
\vphantom{\dfrac12}
\!\!\! \om^{-\frac12} s_1^\R \, \phi(\sqrt2 x)\!\! &  -s_2^\R \, \phi(2 x) &  \!\!\!\om^{\frac12} s_1^\R \, \phi(\sqrt2 x) \!\! &  1
\\
\end{array}
\right)
\!+ O(e^{-4\sqrt2 x})
\end{equation*}
where $\phi(x)=\frac1{2\pi}\int_0^\infty  e^{-x(l+\frac1l)}  \frac{dl}l$.  (The negative signs in (\ref{Gtwo}) disappear because the rays corresponding to those entries are oriented inwards.)
By Laplace's method we have 
$\phi(x)= \tfrac12 (\pi x)^{-\frac12} \, e^{-2x} + O(x^{-\frac32}e^{-2 x})$
as $x\to\infty$.

On the other hand, it can be shown 
(see Appendix C) 
that any radial solution of (\ref{ost}) which is smooth near infinity must satisfy $w_i\to 0$ as $x\to\infty$.  Thus we are in the situation of Proposition \ref{abpositive}, i.e.\ $a,b>0$ here (in fact $a,b\to 1$ as $x\to\infty$), and we have\footnote{Note that all our previous arguments apply equally well if $C$ is replaced by $-C$, but in this calculation the sign of $Y(0)$ would change.  This would change the sign of all $e^{2w_i}$, and we would be in the situation described in the footnote to Proposition \ref{abpositive} where $a$ and $b$ are negative.}
\begin{align*}
Y(0)&=\lim_{\ze\to 0} 
P_\infty^{-1}\ \fPsii_k (\ze,x)\,  e^{-x^2 \ze  d_4 -\frac1\ze  d_4^{-1}}
\\
&=\lim_{\ze\to 0} P_\infty^{-1}\ \Psiz_{\frac74-k}\ \tfrac14C 
e^{ -\frac1\ze  d_4^{-1}}
\\
&=\lim_{\ze\to 0} P_\infty^{-1}\ P_0 (I + O(\ze))  \tfrac14C
\\
&=P_\infty^{-1} P_0 \tfrac14C = \tfrac14 \Om e^{-w}e^{-w} \Om C.
\end{align*}
This is  $\Om e^{-2w}\Om^{-1}$
by formula (F6) of Appendix A.
Hence
\begin{equation*}
Y(0)=
\bp
t_0 & t_1 & t_2 &t_3 \\
t_3 &t_0 &t_1 &t_2 \\
t_2 &t_3 &t_0 &t_1 \\
t_1 &t_2 &t_3 &t_0 
\ep
\end{equation*}
where
\begin{align*}
t_0&=\tfrac14(e^{-2w_0} + e^{2w_0} +e^{-2w_1} +e^{2w_1}) 
= 1+\cdots
\\
t_1&=\tfrac14( e^{-2w_0} + \ii e^{2w_0} -\ii e^{-2w_1} -e^{2w_1}) 
=  -2^{-\tfrac12}\om^{-\tfrac12}(w_0+w_1)+\cdots
\\
t_2&=\tfrac14( e^{-2w_0} - e^{2w_0} -e^{-2w_1} +e^{2w_1})
= w_1-w_0+\cdots
\\
t_3&=\tfrac14( e^{-2w_0} -\ii e^{2w_0} +\ii e^{-2w_1} -e^{2w_1})
= -2^{-\tfrac12}\om^{\tfrac12}(w_0+w_1)+\cdots
\end{align*}
where the dots indicate that higher order terms in $w_0,w_1$ are omitted.
Comparing these with the formula for $Y(0,x)$, and using the asymptotic formula for $\phi(x)$ above,
we obtain:

\begin{theorem}\label{asymptoticsatinfinity}
Let $w_0,w_1,w_2 \, (=\!\!-w_1),w_3 \, (=\!\!-w_0)$ be the solution of 
(\ref{ost}),(\ref{as}) obtained from the Riemann-Hilbert problem (2)${\,}^\prime$.   This solution is smooth near $x=\infty$ and satisfies
\begin{align*}
w_0(x)+w_1(x)\ &=\  -\ s_1^\R\  2^{-\frac34}\  (\pi x)^{-\frac12} \ e^{-2\sqrt2 x}
+ O(x^{-1}e^{-2\sqrt2 x})
\\
w_0(x)-w_1(x)\ &=\  \ \  s_2^\R\  2^{-\frac32}\  (\pi x)^{-\frac12} \ e^{-4 x}
+ O(x^{-1}e^{-4 x})
\end{align*}
as $x\to\infty$.   Thus, if $s_1^\R\ne 0$, we have
\begin{align*}
w_0(x)\ &=\  -\ s_1^\R\  2^{-\frac74}\  (\pi x)^{-\frac12} \ e^{-2\sqrt2 x}+ O(x^{-1}e^{-2\sqrt2 x})
\\
w_1(x)\ &=\  -\ s_1^\R\  2^{-\frac74}\  (\pi x)^{-\frac12} \ e^{-2\sqrt2 x}+ O(x^{-1}e^{-2\sqrt2 x})
\end{align*}
as $x\to\infty$.
If $s_1^\R=0$, we have $w_0+w_1=0$, hence
\begin{align*}
w_0(x)\ &=\  \ \ s_2^\R\  2^{-\frac52}\  (\pi x)^{-\frac12} \ e^{-4 x}+ O(x^{-1}e^{-4 x})
\\
w_1(x)\ &=\  -s_2^\R\  2^{-\frac52}\  (\pi x)^{-\frac12} \ e^{-4 x}+ O(x^{-1}e^{-4 x})
\end{align*}
as $x\to\infty$.
\end{theorem}

As an application of the theorem, we can remove the sign ambiguity from the formula 
$\pm s_1^\R = 2\cos \tfrac\pi4 {\scriptstyle (\ga_0+1)} +  2\cos \tfrac\pi4 {\scriptstyle(\ga_1+3)}$ (Theorem B of \cite{GuItLiXX}) which relates the asymptotic data at $x=0$ to the Stokes data, in the case of solutions which are smooth on $(0,\infty)$:

\begin{corollary}\label{ambiguity}  Let $w_0,w_1,w_2 \, (=\!\!-w_1),w_3 \, (=\!\!-w_0)$ be the solution of 
(\ref{ost}),(\ref{as})
on $(0,\infty)$ which satisfies\footnote{We use the convention of \cite{GuLiXX},\cite{GuItLiXX} that
$2w_i\sim \ga_i \log x$, the factor $2$ coming from the zero curvature representation of (\ref{ost}). This is consistent with $u(x)\sim\ga\log x$, $v(x)\sim\de\log x$ where $\ga=\ga_0$, $\de=\ga_1$
and $u=2w_0$, $v=2w_1$ in the present situation.
} 
$w_0(x)\sim \tfrac12\ga_0\log x$, $w_1(x)\sim \tfrac12\ga_1\log x$ as $x\to 0$, and $w_0(x)\to 0$, $w_1(x)\to 0$ as $x\to \infty$
 (see Theorem B of \cite{GuItLiXX}).   Then we have
\begin{align*}
-s_1^\R &= 2\cos \tfrac\pi4 {\scriptstyle (\ga_0+1)} +  2\cos \tfrac\pi4 {\scriptstyle(\ga_1+3)}
\\
-s_2^\R &= 2+4\cos \tfrac\pi4 {\scriptstyle(\ga_0+1)} \, \cos \tfrac\pi4 {\scriptstyle(\ga_1+3)}.
\end{align*}
\end{corollary}

\begin{proof}  For $\ga_0+\ga_1>0$ we have $w_0(x)+w_1(x)<0$ for all $x\in(0,\infty)$.  
(The case $\ga_0,\ga_1>0$ is Proposition 3.3 of \cite{GuLiXX}. When $\ga_0+\ga_1>0$,
the statement can be proved in the same way, applying the maximum principle to the
the sum of equations (\ref{ost2}), which is 
$2(w_0+w_1)_{z\zbar}= e^{4w_0} - e^{-4w_1}$.)
From Theorem B of \cite{GuItLiXX} we have 
$\pm s_1^\R = 2\cos \tfrac\pi4 {\scriptstyle (\ga_0+1)} +  2\cos \tfrac\pi4 {\scriptstyle(\ga_1+3)}
=4\cos \tfrac\pi8 {\scriptstyle (\ga_0+\ga_1+4)} \cos \tfrac\pi8 {\scriptstyle(\ga_0-\ga_1-2)}$.
It suffices to consider the interior of the region of $(\ga_0,\ga_1)$. Since $\ga_0>-1$, $\ga_1< 1$, $\ga_0-\ga_1< 2$ here, imposing the condition 
$\ga_0+\ga_1>0$ gives
$\tfrac\pi8  (\ga_0+\ga_1+4)\in (\tfrac\pi2,\pi)$ and
$\tfrac\pi8 (\ga_0-\ga_1-2)\in (-\tfrac\pi2,0)$. Hence 
$4\cos \tfrac\pi8 {\scriptstyle (\ga_0+\ga_1+4)} \cos \tfrac\pi8 {\scriptstyle(\ga_0-\ga_1-2)}<0$.
From the previous theorem we have $s_1^\R>0$ in this case, so we must take the negative sign in
Theorem B of \cite{GuItLiXX}.
\end{proof}

Theorem \ref{asymptoticsatinfinity} and Corollary \ref{ambiguity} give the precise \ll connection formula\rr relating the asymptotics of the smooth solutions at zero and infinity. 

Note that if $s_1^\R=0$ then (\ref{ost}) reduces to the sinh-Gordon equation
\[
(w_0)_{\zzb}=\sinh 4w_0.  
\]
In this case Corollary \ref{ambiguity} gives $s_2^\R=-2\sin\tfrac\pi2\ga$.  Thus, for the solution $w_0$ which satisfies $w_0(x)\sim \tfrac\ga2\log x$ as $x\to 0$, we have
\[
w_0(x)\sim -\tfrac{1}{2\sqrt 2}\  \sin\tfrac\pi2\ga \  (\pi x)^{-\frac12} \ e^{-4 x}
\]
as $x\to\infty$. This agrees with the result of \cite{MTW77}
(cf. \cite{FIKN06}, formulae (15.0.1), (15.0.3), (15.0.11)). 

\section{Solution of the Riemann-Hilbert problem for all positive $x$}\label{vanishing}

The \ll Vanishing Lemma\rr (Corollary 3.2 of \cite{FIKN06}) gives a criterion for solvability of the Riemann-Hilbert problem:  {\em If the \ll homogeneous problem\rr  
(in which the condition  $Y\vert_{\ze=\infty}=I$ is replaced by the condition $Y\vert_{\ze=\infty}=0$)
has only the trivial solution $Y=0$, then the original problem is solvable.}

To apply this, it is convenient to simplify the contour $\Ga_2$ in a different way.  Let $\Ga_3$ be the contour obtained by deleting all rays in $\Ga_2$ except for those with arguments $\pi/8,9\pi/8$.  We obtain a simplified Riemann-Hilbert problem on this contour by extending the sectionally-holomorphic functions $\fPsii_{1\scriptstyle\frac12}$, 
$\fPsii_{2\scriptstyle\frac12}$ to the half-planes indicated in 
Fig.\ \ref{simplifiedfolded}.  (This is analogous to the simplified Riemann-Hilbert problem on the contour in Fig.\ \ref{simplified}, which was obtained by extending the sectionally-holomorphic functions $\fPsii_k$ to the interior of the circle.) Note that the original domain of definition of $\fPsii_{1\scriptstyle\frac12}$, 
$\fPsii_{2\scriptstyle\frac12}$ includes these half-planes.
\begin{figure}[h]
\begin{center}
\includegraphics[scale=0.6, trim= 40  260  0  260]{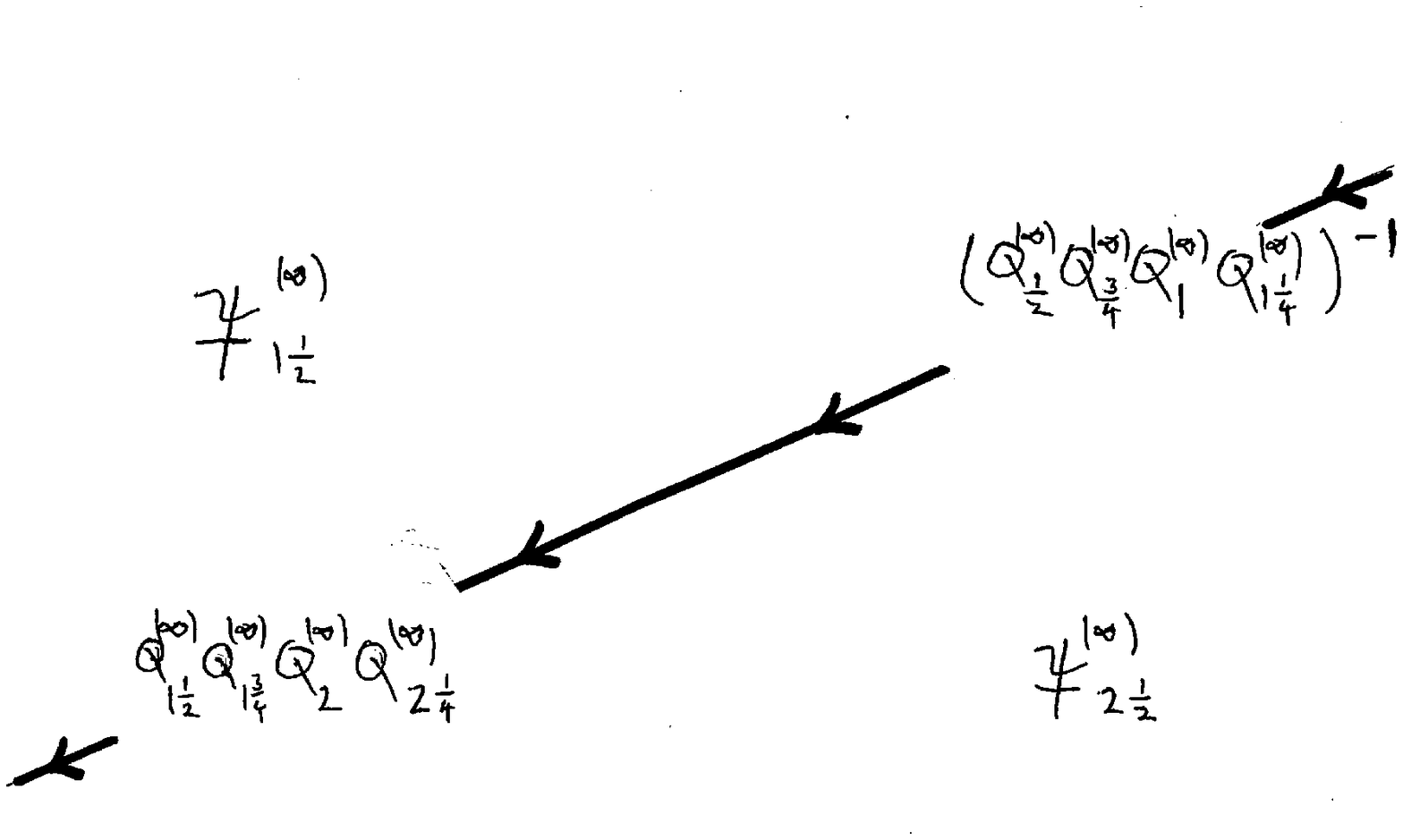}
\end{center}
\caption{The simplified and folded contour $\Ga_3$}\label{simplifiedfolded}
\end{figure}

The jumps $\fPsii_{1\frac12})^{-1} \fPsii_{2\frac12}$ in Fig.\ \ref{simplifiedfolded} can be obtained as follows.  On the $9\pi/8$ ray we have 
$\fPsii_{2\frac12}=\Psii_{2\frac12}$ and $\fPsii_{1\frac12}=\Psii_{1\frac12}$, and $\Psii_{2\frac12} =
\Psii_{1\frac12}\Qi_{1\frac12}\Qi_{1\frac34}\Qi_{2}\Qi_{2\frac14}$. Thus
the jump is $\Qi_{1\frac12}\Qi_{1\frac34}\Qi_{2}\Qi_{2\frac14}$.
On the $\pi/8$ ray we have $\fPsii_{1\frac12}=\Psii_{1\frac12}$, but
$\fPsii_{2\frac12}(\ze)=\Psii_{2\frac12}(e^{2\pi \i}\ze) = 
\Psii_{\frac12}(\ze)$. Now, $\Psii_{1\frac12}=\Psii_{\frac12}
\Qi_{\frac12}\Qi_{\frac34}\Qi_{1}\Qi_{1\frac14}$, so we
conclude that the jump on this ray is
$(\Qi_{\frac12}\Qi_{\frac34}\Qi_{1}\Qi_{1\frac14})^{-1}$.

To apply the theory of \cite{FIKN06}, let us introduce
\begin{gather*}
Y_{-}(\ze,x)=P_\infty^{-1}\ \fPsii_{1\scriptstyle\frac12} (\ze,x)\,  e^{-x^2 \ze  d_4 -\frac1\ze  d_4^{-1}}
\\
Y_{+}(\ze,x)=P_\infty^{-1}\ \fPsii_{2\scriptstyle\frac12} (\ze,x)\,  e^{-x^2 \ze  d_4 -\frac1\ze  d_4^{-1}}.
\end{gather*}
The new jump function is $G_3=e\,\Xi_3 \, e^{-1}$, where $e(\ze,x)=e^{x^2 \ze  d_4 +\frac1\ze  d_4^{-1}}$ and
the matrices $\Xi_3$ are those in
Fig.\ \ref{simplifiedfolded}.

\no{\bf Riemann-Hilbert problem (3):}  {\em 
Let $s_1^\R,s_2^\R\in\R$.  Define
matrices $\Qi_k$ as before.
For these matrices,  find a sectionally-holomorphic function $Y=\{Y_\pm\}$,
such that $Y\to I$ as $\ze\to\infty$, 
whose jumps on the contour $\Ga_3$ are given by $G_3=e\, \Xi_3\, e^{-1}$,
where 
$e(\ze,x)=e^{x^2 \ze  d_4 +\frac1\ze  d_4^{-1}}$ and
$\Xi_3$ is as shown in Fig.\ \ref{simplifiedfolded}.  
}

A solution of the Riemann-Hilbert problem (2) or (2)${\,}^\prime$ gives a solution of (3); conversely, given the matrices $\Qi_k$, a solution of the Riemann-Hilbert problem (3)
gives a solution of (2) or (2)${\,}^\prime$.

The properties\footnote{It is tempting to rotate the contour by $-\pi/8$, to obtain the real line as the new contour $\Ga_3^\prime$.  However the new jump function $G_3^\prime$ would {\em not} have these properties, so we cannot do this.}
$\lim_{\ze\to\infty} G_3(\ze,x)=I$ and 
$\lim_{\ze\to 0} G_3(\ze,x)=I$ can be established in exactly the same way as for $G_2, G_2^\prime$.  We shall give this calculation as we will use the explicit formula for $G_3$ later, for the Vanishing Lemma.   On the ray $\ze=ke^{\i\th}$ (where $\th =\pi/8$ or $9\pi/8$, and $k>0$), we have
\[
(G_3(\ze,x))_{ij}=
\begin{cases}
(\Xi_3)_{ij}\ e^{
\left( 
\frac{2}k e^{\i \boxe}  \ +\    2 k x^2 e^{-\i \boxe}
\right)
}
\ \ \text{if $i-j$ is even}
\\
(\Xi_3)_{ij}\ e^{
\left( 
\frac{\sqrt 2}k e^{\i \boxe}  \ +\    \sqrt 2 k x^2 e^{-\i \boxe}
\right)
}
\ \ \text{if $i-j$ is odd}
\end{cases}
\]
where the boxed angle is indicated in the diagram below:
\[
\begin{array}{|c|c|c|c|}
\hline
\vphantom{\dfrac12}
0 & -\th+\tfrac\pi4 &  -\th+0&   -\th-\tfrac\pi4
\\
\hline
\vphantom{\dfrac12}
 -\th+\tfrac{5\pi}4  & 0 &  -\th-\tfrac{\pi}4 &  -\th-\tfrac{\pi}2
  \\
\hline
\vphantom{\dfrac12}
 -\th+\pi &  -\th+\tfrac{3\pi}4 & 0 &    -\th-\tfrac{3\pi}4
\\
\hline
\vphantom{\dfrac12}
 -\th+\tfrac{3\pi}4 & -\th+\tfrac{\pi}2  &  -\th+\tfrac{\pi}4&  0
\\
\hline
\end{array}
\]
For $\th=9\pi/8$, we have $\Xi_3=
\Qi_{1\frac12}\Qi_{1\frac34}\Qi_{2}\Qi_{2\frac14}$.  Using Appendix A and the above formula for $G_3$, we find
\begin{equation*}
G_3(\ze,x)=
\left(
\begin{array}{c|c|c|c}
\vphantom{\dfrac{A}{A_A}}
 \ \ \ \ 1 \ \ \ \ & \om^{\frac12} (s^\R_1 + s^\R_1s^\R_2)f_1
&  \om^3  ((s^\R_1)^2 + s^\R_2)f_2 &  \ \om^{\frac32} s^\R_1 f_3\ 
\\
\hline
\vphantom{\dfrac{A}{A_A}}
 & 1 &  \om^{\frac12} s^\R_1 f_3 & 
  \\
\hline
\vphantom{\dfrac{A}{A_A}}
 &  &  1 &  
\\
\hline
\vphantom{\dfrac{A}{A_A}}
 &  \om^3 s^\R_2 f_4 &  \om^{\frac32} s^\R_1 f_1 &  1
\\
\end{array}
\right)
\end{equation*}
where 
\begin{align*}
f_1&=e^{
\sqrt2\left(
\frac1k e^{(-9+2)\pi\i/8}  +  x^2ke^{(9-2)\pi\i/8} 
\right)
}
\\
f_2&=e^{
2\left(
\frac1k e^{(-9+0)\pi\i/8}  +  x^2ke^{(9-0)\pi\i/8} 
\right)
}
\\
f_3&=e^{
\sqrt2\left(
\frac1k e^{(-9-2)\pi\i/8}  +  x^2ke^{(9+2)\pi\i/8} 
\right)
}
\\
f_4&=e^{
2\left(
\frac1k e^{(-9+4)\pi\i/8}  +  x^2ke^{(9-4)\pi\i/8} 
\right).
}
\end{align*}
Similarly, for $\th=\pi/8$, we have
\begin{equation*}
G_3(\ze,x)=
\left(
\begin{array}{c|c|c|c}
\vphantom{\dfrac{A}{A_A}}
 1 & 
&   &  
\\
\hline
\vphantom{\dfrac{A}{A_A}}
\om^{-\frac12} (s^\R_1 + s^\R_1s^\R_2)g_1  & 1 &  & 
 \om s^\R_2 g_4
\\
\hline
\vphantom{\dfrac{A}{A_A}}
\om  ((s^\R_1)^2 + s^\R_2)g_2 &  \om^{-\frac12} s^\R_1 g_3 &   \ \ \ \ 1  \ \ \ \ &  
 \om^{-\frac32} s^\R_1 g_1
\\
\hline
\vphantom{\dfrac{A}{A_A}}
\ \om^{-\frac32} s^\R_1 g_3\  &   &   &  1
\\
\end{array}
\right)
\end{equation*}
where 
\begin{align*}
g_1&=e^{
\sqrt2\left(
\frac1k e^{(-1+10)\pi\i/8}  +  x^2ke^{(1-10)\pi\i/8} 
\right)
}
=f_1
\\
g_2&=e^{
2\left(
\frac1k e^{(-1+8)\pi\i/8}  +  x^2ke^{(1-8)\pi\i/8} 
\right)
}
=f_2
\\
g_3&=e^{
\sqrt2\left(
\frac1k e^{(-1+6)\pi\i/8}  +  x^2ke^{(1-6)\pi\i/8} 
\right)
}
=f_3
\\
g_4&=e^{
2\left(
\frac1k e^{(-1-4)\pi\i/8}  +  x^2ke^{(1+4)\pi\i/8} 
\right)
}
=f_4.
\end{align*}
Since $f_i\to 0$ as $k\to 0$ or $\infty$, we have 
$\lim_{\ze\to\infty} G_3(\ze,x)=I$ and 
$\lim_{\ze\to 0} G_3(\ze,x)=I$, as required.

To apply the Vanishing Lemma, we need:

\begin{proposition}  Let $Y_\pm$ be a solution of the homogeneous problem for the contour $\Ga_3$. Then:

\no(a) $\displaystyle \int_{\Ga_3} Y_{+}(\ze,x) \ 
\overline{
Y_{-}(\bar\ze e^{2\pi\i/8},x)
}
^{\, t}\  d\ze=0$

\no(b) $\displaystyle \int_{\Ga_3} Y_{-}(\ze,x) \ 
\overline{
Y_{+}(\bar\ze e^{2\pi\i/8},x)
}
^{\, t}\  d\ze=0$
\end{proposition}

\begin{proof}   (a) The function 
$Y_{+}(\ze,x) \ 
\overline{
Y_{-}(\bar\ze e^{2\pi\i/8},x)
}
^{\, t}$ is holomorphic on the lower region (since $Y_{-}$, $Y_{+}$ are
holomorphic on the upper/lower regions, respectively).  Since $G_3\to I$ exponentially as $\ze\to\infty$, the same is true for $Y$, so the stated result follows from Cauchy's Theorem.   The proof of (b) is similar.
\end{proof}

\begin{corollary} Fix $x\in(0,\infty)$.   If $G_3(\ze,x)+\overline{
G_3(\bar\ze e^{2\pi\i/8},x)
}
^{\, t}$
is a positive definite matrix for every $\ze$ on the contour $\Ga_3$, then the homogeneous
Riemann-Hilbert problem on the contour $\Ga_3$ has only the trivial solution $Y=0$.  Hence, by the Vanishing Lemma,  the original 
Riemann-Hilbert problem on the contour $\Ga_3$ is solvable (for the given value of $x$).
\end{corollary}

\begin{proof}  Substitute $Y_+=Y_-G_3$ and add formulae (a) and (b) of the proposition.
(Note that $\bar\ze e^{2\pi\i/8}=\ze$ on the contour $\Ga_3$.)
\end{proof}

We shall obtain a criterion for positive-definiteness which depends on the Stokes data $s_1^\R,s_2^\R$.  The criterion for the $9\pi/8$ ray turns out to be the same as the criterion for the $\pi/8$ ray, so we shall just give details for the latter.  

Let $X(\ze,x)=G_3(\ze,x)+\overline{G_3(\bar\ze e^{2\pi\i/8},x)}^{\, t}$, and write $X=(X_{ij})_{1\le i,j\le 4}$,  $X_k=\det (X_{ij})_{5-k\le i,j\le 4}$.  Since $X_1=2$, we have:
$X$ is positive definite if and only if $X_2,X_3,X_4>0$.

For $\ze=ke^{\i \pi/8}$ we have $X(\ze,x)=$
\begin{equation*}
\left(
\begin{array}{c|c|c|c}
\vphantom{\dfrac{A}{A_A}}
 2 & \om^{\frac12} (s^\R_1 + s^\R_1s^\R_2)\bar g_1
&  \om^{-1}  ((s^\R_1)^2 + s^\R_2)\bar g_2 &  \om^{\frac32} s^\R_1 \bar g_3
\\
\hline
\vphantom{\dfrac{A}{A_A}}
\!\!\om^{-\frac12} (s^\R_1 + s^\R_1s^\R_2)g_1\!\!  & 2 & \om^{\frac12} s^\R_1 \bar g_3  & 
 \om s^\R_2 g_4
\\
\hline
\vphantom{\dfrac{A}{A_A}}
\om  ((s^\R_1)^2 + s^\R_2)g_2 &  \om^{-\frac12} s^\R_1 g_3 &   \ \ \ \ 2  \ \ \ \ &  
\  \om^{-\frac32} s^\R_1 g_1 \ 
\\
\hline
\vphantom{\dfrac{A}{A_A}}
\om^{-\frac32} s^\R_1 g_3  &  \om^{-1} s^\R_2 \bar g_4  &   \om^{\frac32} s^\R_1 \bar g_1 &  2
\\
\end{array}
\right)
\end{equation*}
and this gives
\begin{align*}
X_2&=
4 - \vert g_1\vert^2 (s_1^\R)^2
\\
X_3&=
8-2(s_1^\R)^2 s_2^\R \vert g_1\vert^2 -2(s_1^\R)^2 ( \vert g_1\vert^2 + \vert g_3\vert^2 ) - 2(s_2^\R)^2 \vert g_4\vert^2
\\
X_4&=16  -  8(s_1^\R)^2 ( \vert g_1\vert^2 + \vert g_3\vert^2 )
-  8(s_1^\R)^2 s_2^\R ( \vert g_1\vert^2 + \vert g_2\vert^2 )
- 4 (s_2^\R)^2 ( \vert g_2\vert^2 + \vert g_4\vert^2 )
\\
&
-  2(s_1^\R)^2 (s_2^\R)^2 ( \vert g_2\vert^2 + \vert g_1\vert^4 )
+  (s_1^\R)^4 ( \vert g_1\vert^4 + \vert g_3\vert^4 - 2\vert g_2\vert^2 )
+  (s_2^\R)^4 \vert g_1\vert^4.
\end{align*}
Now, we have
\begin{align*}
\vert g_1\vert^2 &= e^{ -2x(l+\frac1l)(\cos\pi/8+\sin\pi/8)}
= e^{ -2\sqrt2 x(l+\frac1l)\cos\pi/8}
\\
\vert g_2\vert^2 &= e^{ -4x(l+\frac1l)\cos\pi/8} 
\\
\vert g_3\vert^2 &= e^{ -2x(l+\frac1l)(\cos\pi/8-\sin\pi/8)} 
= e^{ -2\sqrt2 x(l+\frac1l)\sin\pi/8} 
\\
\vert g_4\vert^2 &= e^{ -4x(l+\frac1l)\sin\pi/8}.
\end{align*}
where $l=xk$.  The region 
\[
\{ (s_1^\R,s_2^\R)\in\R^2 \st 
X_2,X_3,X_4>0 \}
\]
shrinks as $x(l+\frac1l)$ decreases.  
Since $\{ l+\frac1l \st l>0 \} = [2,\infty)$,  we obtain:

\begin{theorem}\label{xregion}  Let $w_0,w_1,w_2 \, (=\!\!-w_1),w_3 \, (=\!\!-w_0)$ be the solution of 
(\ref{ost}),(\ref{as})
obtained from the Riemann-Hilbert problem (3) (or (2), (2)${\,}^\prime$).  This solution is 
smooth on $(x,\infty)$  if $(s_1^\R,s_2^\R)$ satisfies the conditions
\begin{align*}
0&<4-cs(s_1^\R)^2
\\
0&<8-2(cs+c/s)(s_1^\R)^2-2s^2(s_2^\R)^2-2cs(s_1^\R)^2(s_2^\R)
\\
0&<16-8(cs+c/s)(s_1^\R)^2-4(c^2+s^2)(s_2^\R)^2-8(cs+c^2)(s_1^\R)^2(s_2^\R)
\\
&+(c^2/s^2+c^2s^2-2c^2)(s_1^\R)^4
-2(c^2+c^2s^2)(s_1^\R)^2(s_2^\R)^2
+c^2s^2(s_2^\R)^4
\end{align*}
where $c=e^{-4x\cos\pi/8}$, $s=e^{-4x\sin\pi/8}$.    
\end{theorem} 

We recall that $w_0,w_1,w_2,w_3$ were introduced in Proposition \ref{abpositive}.  The fact that $a,b>0$ on the interval $(x,\infty)$ follows from the fact that $a,b>0$ near $\infty$
and, if there were some $x_0\in(x,\infty)$ with $a(x_0)=0$ or $b(x_0)=0$ then $Y(x_0)$ would fail to be invertible, a contradiction.  

The case $x=0$ gives our main conclusion:

\begin{theorem}\label{positivity}   Let $w_0,w_1,w_2 \, (=\!\!-w_1),w_3 \, (=\!\!-w_0)$ be the solution of 
(\ref{ost}),(\ref{as})
obtained from the Riemann-Hilbert problem (3) (or (2),  (2)${\,}^\prime$).  This solution is 
smooth on $(0,\infty)$  if $(s_1^\R,s_2^\R)$ satisfies the conditions
\begin{equation}\label{triangle}
\text{$2+s_2^\R>0$, $2+2s_1^\R-s_2^\R>0$, $2-2s_1^\R-s_2^\R>0$.}
\end{equation}
\end{theorem}

\begin{proof} When $x=0$ we have $c=s=1$.  In this case, the right hand sides of the inequalities in Theorem \ref{xregion} factor as follows:
\begin{align*}
&(2-s_1^\R)(2+s_1^\R)
\\
&(2+s_2^\R)(2-(s_1^\R)^2-s_2^\R)
\\
&(2+s_2^\R)^2(2+2s_1^\R-s_2^\R)(2-2s_1^\R-s_2^\R).
\end{align*}  
The region where these are simultaneously positive is the region given by 
$2+s_2^\R>0$, $2+2s_1^\R-s_2^\R>0$, $2-2s_1^\R-s_2^\R>0$.
This can be seen by direct calculation; we shall give a more conceptual argument later,  in the
proof of Theorem \ref{linalg}.
\end{proof}

This region is the (interior of the) shaded region in Fig.\ \ref{results4}.  The dots represent the solutions with integer Stokes data, which were discussed in detail in \cite{GuLi2XX}.
The region for which smooth solutions exist was determined in \cite{GuLiXX},\cite{GuItLiXX}; it is the closed region bounded by the heavy lines and curve, i.e.\
\begin{equation}\label{all}
\text{
$(s_1^\R)^2  + 4s_2^\R + 8\ge 0$, $2+2s_1^\R-s_2^\R\ge 0$, $2-2s_1^\R-s_2^\R\ge 0$.
}
\end{equation}
Thus, the above calculation gives a proper open subset of this region.  For this open subset, the Riemann-Hilbert method gives an alternative proof\footnote{The solutions constructed by the Riemann-Hilbert are radial solutions.  
By Appendix C, 
such solutions necessarily satisfy asymptotic boundary conditions $w_i(x)/\log x\to\text{constant}$, as $x\to 0$, and $w_i(x)\to 0$, as $x\to\infty$.} 
(to the p.d.e.\ proof in 
\cite{GuLiXX},\cite{GuItLiXX}) of the existence and uniqueness of solutions of 
(\ref{ost}),(\ref{as}) parametrized by Stokes data $(s_1^\R,s_2^\R)$
or asymptotic data $(\ga_0,\ga_1)$.  
\begin{figure}[h]
\begin{center}
\includegraphics[scale=0.4, trim= 40  120  0  120]{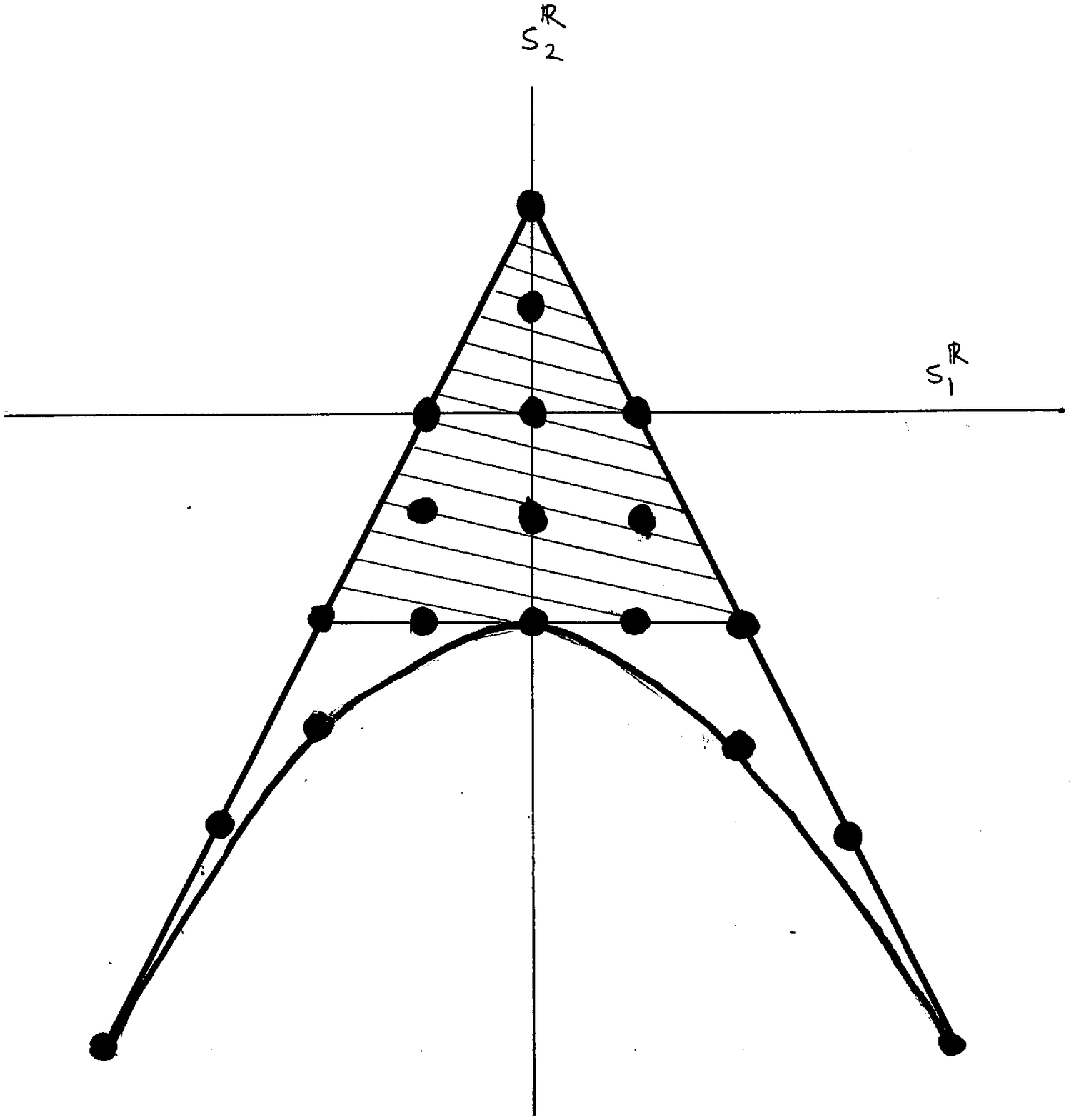}
\end{center}
\caption{Solutions of the tt*-Toda equations (case 4a).}\label{results4}
\end{figure}
As the monodromy data depends analytically on the coefficients of the meromorphic differential equation, we can deduce that the connection matrix $E_1$ (which we have computed in this article only for the above open subset of $w_0,w_1$) is in fact given by
the same formula $E_1=\tfrac14 C \Qi_{\frac34}$ for all smooth solutions. Combining this fact with Corollary} \ref{ambiguity}, we can now give a complete statement of the monodromy data:

\begin{theorem}\label{alldata}
Let $w_0,w_1,w_2 \, (=\!\!-w_1),w_3 \, (=\!\!-w_0)$ be a solution
of the tt*-Toda equations which is smooth  
on $(0,\infty)$.  Then the Stokes data $(s_1^\R,s_2^\R)$  of the associated meromorphic o.d.e.\  (\ref{ode}) is given by
\begin{align*}
s_1^\R &= -2\cos \tfrac\pi4 {\scriptstyle (\ga_0+1)} -  2\cos \tfrac\pi4 {\scriptstyle(\ga_1+3)}
\\
s_2^\R &= -2-4\cos \tfrac\pi4 {\scriptstyle(\ga_0+1)} \, \cos \tfrac\pi4 {\scriptstyle(\ga_1+3)}
\end{align*}
and the connection matrix is
$E_1=\tfrac14 C \Qi_{\frac34}$ where 
$C=
\bsp
1\! & & & \\
 & & & \!1\\
 & & \!1\! & \\
 & \!1\! & &
\esp.
$
We recall that the correspondence between $(w_0,w_1)$ and $(\ga_0,\ga_1)$ is given by
 $w_0(x)\sim \tfrac12\ga_0\log x$, $w_1(x)\sim \tfrac12\ga_1\log x$ as $x\to 0$.  
\end{theorem}

The special nature of the connection matrix (essentially just the constant matrix $C$) is to be expected, as our solutions are generalizations of the smooth solutions of the sinh-Gordon equation obtained by the Riemann-Hilbert method in \cite{FIKN06}.  In general, four real
parameters would be needed to describe arbitrary (locally defined) solutions, the two Stokes parameters $(s_1^\R,s_2^\R)$ and two additional parameters in the connection matrix.  Our calculation shows that the Stokes parameters determine these additional parameters, for the solutions which are globally smooth on $\C^\ast$.

The geometry of the moduli space of (locally defined) solutions is complicated.  It has been investigated for the sinh-Gordon equation in \cite{FIKN06},\cite{ItNi11} by Riemann-Hilbert methods and by other authors using different methods.  Little is known for the tt* equations in general.  Even if attention is restricted to the \ll globally smooth\rr solutions, it is a subtle matter to describe these in terms of the monodromy data.  The work of Cecotti and Vafa has inspired a number of conjectures and results in this direction, notably in  \cite{Sa08},\cite{HeSe07},\cite{HeSa11}.  In particular, 
Conjecture 10.2 of \cite{HeSe07}, namely 
the smoothness criterion \ll $S+S^t>0$\rr  (where $S$ is the Stokes matrix), was
established in \cite{HeSa11}.   This general result implies our Theorem \ref{positivity}.  
However, our approach (combined with our previous work \cite{GuLiXX},\cite{GuItLiXX})
gives a necessary and sufficient smoothness condition in our situation (and in fact for the tt*-Toda equations in general).  We shall explain this next.

At this point it will be convenient to adjust our conventions so that our Stokes matrices are real.  This can be achieved by using
\[
\tilde P_0=P_0 \,\dz,\quad 
\tilde P_\infty=P_\infty\,\di
\]
instead of $P_0,P_\infty$, where
$\dz=\diag(1,\om^{\frac12},\om,\om^{\frac32})=\di^{-1}$.  We obtain formal solutions
$\tPsiz=\Psiz \dz, \tPsii=\Psii\di$, i.e.\ 
\begin{gather*}
\tPsiz_f = P_0\left( I + \sum_{k\ge 1} \psiz_k \ze^k \right) e^{\frac1\ze  d_4}\dz=
\tilde P_0\left( I + \sum_{k\ge 1} \tpsiz_k \ze^k \right) e^{\frac1\ze  d_4}
\\
\tPsii_f = P_\infty\left( I + \sum_{k\ge 1} \psii_k \ze^{-k} \right) e^{x^2 \ze  d_4}\di
=\tilde P_\infty\left( I + \sum_{k\ge 1} \tpsii_k \ze^{-k} \right) e^{x^2 \ze  d_4}.
\end{gather*}
The matrices
\[
\tQz_k=\dz^{-1} \Qz_k \dz, \quad \tQi_k=\di^{-1} \Qi_k \di
\]
are now real.  Furthermore, they satisfy $\tQz_k=\tQi_k$ (by Lemma \ref{zandi}), and have the following symmetries (by  Lemma \ref{stokessymmetries}):
 
\no
(1) $\tQi_{k+\scriptstyle\frac12} = \tPi\, \tQi_k \, \tPi^{-1}$,\ 
$\tPi=\di^{-1} \Pi\di=
\om^{-\frac12}
\bsp
  &  1  & & \\
 & &  1  & \\
  & & &  1\\
-1    & & &
\esp$

\no
(2) $\tQi_{k+1} = \tQi_k{}^{-t}$

\no
(3) $\tQi_k= \tilde C \overline{\tQi_{ {\scriptstyle\frac32} - k}  } {}^{-1} \tilde C$,\ 
$\tilde C=\di^{-1}  C \di=
\bsp
1\! & & & \\
 & & & \!-1\!\\
 & & \!-1\! & \\
 & \!-1\! & &
\esp
$

\no The connection matrix
$\tilde E_1=\tfrac14 \tQi_{\frac34} \tilde C$ is also real.  The Stokes matrices 
$\tSi_k=\di^{-1} \Si_k \di$ satisfy
\[
\tSi_1=(\tQi_1 \tQi_{1\frac14}\tPi)^2 \tPi^2,  \quad
\tSi_2= \tPi^2 (\tQi_1 \tQi_{1\frac14}\tPi)^2 = (\tSi_1)^{-t}.
\]
To simplify notation, let us write
\[
S=\tSi_1,\quad \hat \Pi=\bsp
  &  1  & & \\
 & &  1  & \\
  & & &  1\\
-1    & & &
\esp
\]
from now on.  Then the monodromy is
\[
S S^{-t} =  (\tQi_1 \tQi_{1\frac14}\tPi)^4
= - (\tQi_1 \tQi_{1\frac14} \hat \Pi)^4.
\]  
The characteristic polynomial of $\tQi_1 \tQi_{1\frac14} \hat \Pi$
is 
\[
p(\mu)=\mu^4 + s_1^\R \mu^3 - s_2^\R \mu^2 + s_1^\R \mu + 1.  
\]
In the Riemann-Hilbert problem (3) we just replace the jumps by their tilde versions. These are:  $\tPi^{-1} S^{-1} \tPi$ on the $\pi/8$ ray, and 
$\tPi\, S\, \tPi^{-1}$ on the $9 \pi/8$ ray.  
The tilde version of Theorem \ref{positivity} is that we have a solution for all $x$ in $(0,\infty)$ if the Hermitian matrices
$\tPi^{-1} (S^{-1}+S^{-t})\tPi$ and $\tPi (S + S^t)\tPi^{-1}$
are positive definite.  Since $\tPi$ is unitary, this is equivalent to the real symmetric matrices
$S^{-1}\!+\!S^{-t}$ and $S + S^t$ being positive definite.  Furthermore, because of the identities
\[
S^{-1}\!+\!S^{-t}=S^{-1}(I+SS^{-t}),\quad
S+S^t = (SS^{-t} + I)S^t,
\]
$S^{-1}\!+\!S^{-t}$ is positive definite if and only if
$S + S^t$ is positive definite (both are equivalent to all principal minors of $SS^{-t}+I$ being positive definite).  Thus our criterion (Theorem \ref{positivity}) coincides with the criterion of \cite{HeSa11}, namely
$S + S^t>0$.

It was pointed out already by Cecotti and Vafa that, if $S^{-1} + S^{-t}>0$, then the monodromy $SS^{-t}$ preserves the positive definite inner product defined by $S^{-1} + S^{-t}$, hence the
eigenvalues of $SS^{-t}$ must have unit length.  As we have shown in \cite{GuLiXX},\cite{GuItLiXX}, this condition on the eigenvalues is satisfied whenever the solution of 
(\ref{ost}),(\ref{as}) is smooth near $x=0$.  In particular, this is a necessary condition for smoothness of the solution on $(0,\infty)$. In our case, it turns out to be a sufficient condition. We summarize all this in the following theorem, which also provides a conceptual explanation for the explicit formulae 
(\ref{triangle}),(\ref{all}) for the regions illustrated in Fig.\ \ref{results4}. 

\begin{theorem}\label{linalg}   Let $w_0,w_1,w_2 \, (=\!\!-w_1),w_3 \, (=\!\!-w_0)$ be the solution of 
(\ref{ost}),(\ref{as})
obtained from the Riemann-Hilbert problem (3) (or (2),  (2)${\,}^\prime$).  Then we have:

\no (a) The following are equivalent:

(i) $w_0,w_1,w_2, w_3$ are smooth on $(0,\infty)$,

(ii) all roots of $p$ lie in the unit circle,

(iii) all eigenvalues of $SS^{-t}$ lie in the unit circle,

(iv) $(s_1^\R)^2  + 4s_2^\R + 8\ge 0$, $2+2s_1^\R-s_2^\R\ge 0$, $2-2s_1^\R-s_2^\R\ge 0$.

\no Moreover, all radial solutions of (\ref{ost}),(\ref{as}) (in case 4a) which are smooth on $\C^\ast$ are of this form.

\no (b) The following are equivalent:

(i) $S^{-1}\!+\!S^{-t}>0$,

(ii) $p(\om^i)>0$ for $i=0,1,2,3$,

(iii) $2+s_2^\R>0$, $2+2s_1^\R-s_2^\R>0$, $2-2s_1^\R-s_2^\R>0$.

\no  In Fig.\ \ref{results4}, conditions (a) give the closed region bounded by the heavy lines and curve, and
conditions (b) give the (interior of the) shaded region.
\end{theorem}

\begin{proof} (a) The region of the $(\ga_0,\ga_1)$-plane given by Theorem A of \cite{GuItLiXX} is
\begin{equation}\label{pderegion}
-1\le \ga_0\le 3,\ -3\le \ga_1\le 1,\ \ga_0-\ga_1\le 2.
\end{equation}
Theorem \ref{alldata} gives the corresponding region (i) of the $(s_1^\R,s_2^\R)$-plane.  It follows from Appendix C that these (globally smooth) solutions are a subset of the (smooth near infinity) solutions which arise from the Riemann-Hilbert problem (3) (or (2),  (2)${\,}^\prime$.

To prove the equivalence of (i) and (ii),   observe that
\[
\mu^{-2} p(\mu)=(\mu +\mu^{-1})^2 + s_1^\R(\mu +\mu^{-1}) - (2+s_2^\R).
\] 
It follows that all roots $\mu$ of $p$ lie in the unit circle if and only if all roots $x$ of the quadratic
\[
P(x)=x^2 + s_1^\R x - (2+s_2^\R)
\]
lie in the interval $[-2,2]$, i.e.\ are of the form $2\cos\th_1,2\cos\th_2$ for
some $\th_1,\th_2\in\R$.  By Theorem \ref{alldata},  the points of (i) are indeed of this form,
with $\th_1=\pm\tfrac\pi4 { (\ga_0+1)}$, 
$\th_2=\pm\tfrac\pi4 { (\ga_1+3)}$.  Thus, (i) implies (ii). On the other hand, 
it is easy to verify that the region
\[
\{(\th_1,\th_2)\in\R^2 \st 0\le \th_1,\th_2 \le \pi \text{ and } \th_1\le \th_2 \}
\]
is a fundamental domain for the (branched) covering map
\[
(\th_1,\th_2)\mapsto (\cos\th_1+\cos\th_2,\cos\th_1\cos\th_2).  
\]
Our region (\ref{pderegion}) is exactly this region. Thus (i) and (ii) are equivalent.

Next, from $S S^{-t} = - (\tQi_1 \tQi_{1\frac14} \hat \Pi)^4$,
it is clear that (ii) is equivalent to (iii).   It remains to establish the explicit description (iv).  Let us use the criterion (above) that all roots of $P(x)=x^2 + s_1^\R x - (2+s_2^\R)$ lie in the interval $[-2,2]$.  This is equivalent to  (1) the condition that $P$ has only real roots, i.e.\ 
$(s_1^\R)^2+4(2+s_2^\R)\ge 0$, together with (2) the condition that these roots lie in the interval $[-2,2]$, which (as $P$ is quadratic) means $P(-2)\ge0$, $P(2)\ge0$, i.e.\ 
$2-2s_1^\R-s_2^\R\ge 0$, $2+2s_1^\R-s_2^\R\ge 0$.   Conditions (1),(2) together give the region (iv).

(b) First we note that 
$
\det(S^{-1}+S^{-t}) = \det S^{-1} \det(I+ SS^{-t})
= \det( SS^{-t} + I)= 
 \det(I- (\tQi_1 \tQi_{1\frac14} \hat \Pi)^4).
 $
The identity
\[
X^4 - I  =  (X-1)(X-\om)(X-\om^2)(X-\om^3),
\]
with $X=\tQi_1 \tQi_{1\frac14} \hat \Pi$,  shows that
\[
\det(   (\tQi_1 \tQi_{1\frac14} \hat \Pi)^4 - I )
=
p(1)p(\om)p(\om^2)p(\om^3).
\]
Thus
\[
\det(S^{-1}+S^{-t})=
(2+2s_1^\R-s_2^\R) (2+s_2^\R)(2-2s_1^\R-s_2^\R) (2+s_2^\R).
\]
Let us assume that (i) holds.  Then the monodromy $SS^{-t}$ preserves the positive definite inner product defined by $S^{-1}\!+\!S^{-t}$, hence the
eigenvalues of $SS^{-t}$ must have unit length, and we are in the situation of (a).  From the proof of (a), the roots of $p$ are of the form $e^{\pm\i \th_1}, e^{\pm\i \th_2}$ with $0\le \th_1\le \th_2 \le\pi$.  
Since $\det(S^{-1}\!+\!S^{-t})>0$, we have $p(\om^i)\ne 0$ for all $i$.  Explicitly, 
\begin{align*}
p(\om^i)
&=(\om^i - e^{\i \th_1})(\om^i - e^{-\i \th_1})(\om^i - e^{\i \th_2})(\om^i - e^{-\i \th_2})
\\
&=
\begin{cases}
-4\cos\th_1 \cos\th_2 \quad\text{ if $i=1,3$}
\\
(2\pm 2\cos\th_1)(2\pm 2\cos\th_2) \quad\text{ if $i=0,2$.}
\end{cases}
\end{align*}
Thus $p(\om^0),p(\om^2)>0$ and it suffices to examine  $p(\om^1) =p(\om^3)$.
If $p(\om^1)<0$ then $\cos\th_1$ and $\cos\th_2$ have the same sign, so
$0<\th_1\le\th_2<\tfrac\pi2$ or $\tfrac\pi2<\th_1\le\th_2<\pi$.  Then
we may move both of $e^{\i \th_1}, e^{\i \th_2}$ continuously to  $e^{\pi \i/ 4 }$, or
to $e^{ 3\pi \i/ 4}$, without violating the condition $\det(S^{-1}\!+\!S^{-t})>0$.
It is easily checked that $S^{-1}\!+\!S^{-t}$ is not positive definite for these values.
We conclude that (ii) holds.

Conversely, if (ii) holds, then the linear inequalities $p(\om^i)>0$ define a (convex) connected region of the $(s_1^\R,s_2^\R)$-plane, and it suffices to check that $S^{-1}\!+\!S^{-t}$ is positive definite for at least one point of this region.  The region contains $(s_1^\R,s_2^\R)=(0,0)$,  and for this point we have $S=I$, so $S^{-1}\!+\!S^{-t}$ is positive definite on the entire region, i.e.\ (i) holds.

This completes the proof of the equivalence of (i) and (ii).  From the formula $p(\mu)=\mu^4 + s_1^\R \mu^3 - s_2^\R \mu^2 + s_1^\R \mu + 1$ we obtain (iii).
\end{proof}

\begin{remark}\label{remarks}
(1) 
It follows from the above description of the roots of $p$ that (the bounding curves of) region (a) can be obtained from the discriminant
\[
\text{disc}(p)=
((s_1^\R)^2  + 4s_2^\R + 8)^2(2+2s_1^\R-s_2^\R)(2-2s_1^\R-s_2^\R).
\]

\no(2) The proof of Theorem \ref{linalg} shows that the region (b) may be characterized as the subregion of (a) for which the roots $e^{\pm\i \th_1}, e^{\pm\i \th_2}$ of $p$ interlace with the roots of unity $1,\om,\om^2,\om^3$ (cf.\ \cite{BeHe89}, Corollary 4.7, where a similar criterion is given by Beukers and Heckman in the context of the hypergeometric equation). 

\no(3) When the roots of $p$ are in the unit circle, i.e. $(s_1^\R,s_2^\R)$ is in the region (a), and the $s_1^\R,s_2^\R$ are integers, an observation of Kronecker (\cite{Kr1857}) implies that $p$ must be a product of cyclotomic polynomials.  Conversely, any product of cyclotomic polynomials which is
palindromic and of degree $4$ has the form of $p$ with $s_1^\R,s_2^\R$ integers.  There are exactly 19 such polynomials, and these correspond to the 19 points of the region (a) which
correspond to \ll physical solutions\rr of the tt*-Toda equations (cf.\ \cite{GuLi2XX}).  This is, in fact, the original approach suggested by Cecotti and Vafa for the classification of physical solutions of the tt* equations.  Theorem \ref{linalg} justifies this approach in the case of the tt*-Toda equations.
 
\no(4) Theorem \ref{linalg} (a) links the tt*-Toda equations with the Poisson geometry of spaces of meromorphic connections (cf.\ \cite{Bo01}).  In particular, the convexity of the region in the $(\ga_0,\ga_1)$-plane can be understood in these terms.  This will be discussed in \cite{GuHoXX}.

\no(5) The method of proof of Theorem \ref{linalg} applies to the tt*-Toda equations in general.    For (a) this follows from the fact that the characteristic polynomial of $\tQi_1 \tQi_{1\scriptstyle\frac1{n+1}} \Pi$ is palindromic (or anti-palindromic), because of the anti-symmetry condition.  Namely, any (real) palindromic polynomial factors into quartic (and possibly quadratic or linear) palindromic factors, and the arguments above apply.  This description of the roots of $p$ allows one to deduce (b), as in the proof above.
Remarks (1)-(4) extend also to the general case, thanks to this description.
\end{remark}

In Appendix B we summarize the formulae for the remaining cases covered by Theorem A of \cite{GuItLiXX}, which are very similar.

\section{The Fredholm determinant approach of Tracy-Widom}\label{tracywidom}

When $w_0,\dots,w_n$ reduce to one unknown function, the tt*-Toda equations
are the radial sinh-Gordon equation $w_{z\zbar}=\sinh w$ or the radial 
Tzitzeica (Bullough-Dodd) equation $w_{z\zbar}=e^{w}-e^{-2w}$.   Cecotti and Vafa were able to
analyze the \ll physical solutions\rr in this situation by appealing to the
pioneering work of McCoy, Tracy, and Wu \cite{MTW77} and Kitaev \cite{Ki89}.
A simpler and more general approach to equations of this type was given subsequently by Tracy and Widom,
and in
\cite{TrWi98} they studied a class of solutions to (\ref{ost}).  Although they did not identify this
with the class of solutions with are smooth on $(0,\infty)$,
they gave explicit formulae for the asymptotics of these solutions at $0$ and $\infty$.  
In this section we explain briefly the relation with our approach.  We shall discuss the asymptotics more thoroughly in a separate article.

The equations of \cite{TrWi98} are
\begin{equation}\label{twost}
\tfrac14(q_k^\prr(t)+t^{-1}q_k^\pr(t))=-e^{q_{k+1}-q_k} + e^{q_{k}-q_{k-1}}
\end{equation}
with $q_k=q_{k+N}$ and $q_k=q_k(t)$, where $t\in(0,\infty)$. 
The solutions in the above class are expressed in terms of Fredholm determinants
\[
q_k=\log\det (I-\la K_k) - \log\det (I-\la K_{k-1})
\]
where the operators $K_k$ are defined by
\[
K_k(f)(u)=\int_0^\infty K_k(u,v) f(v) dv
\]
and the kernel $K_k(u,v)$ is
\[
K_k(u,v)=\sum_{j=1}^N \om_j^k c_j \frac{e^{-t[(1-\om_j)u+(1-w_j^{-1})u^{-1}]}}{-\om_j u + v}.
\]
Here, $\om_j=(e^{2\pi\i/N})^j$ ($1\le j\le N$)  and $c_1,\dots,c_N$ are complex parameters.
The condition $c_{\omega} = -\omega^{3}c_{\omega^{-1}}$  corresponds to the condition
$q_k+q_{N-k-1}=0$, so let us impose this.  

According to \cite{TrWi98}, these solutions have the asymptotics
\begin{equation}\label{twasymp}
q_k(t)\sim 2(\al_k - k)\log t \quad \text{ as $t\to 0$}
\end{equation}
where $\al_1,\dots,\al_N$ are described in terms of $c_1,\dots,c_N$ as follows:
the real numbers
$\al_1(\la),\dots,\al_N(\la)$ are the zeros of the function
\[
h(s)=\sin\pi s -\la \pi \sum_{j=1}^N c_j (-\om_j)^{\al-1},
\]
where it is {\em assumed} that there exists a continuous path 
$\{ \la_t \st 0\le t\le 1\}$ from $\la_0=0$ to $\la_1=1$ with
the properties

\no(P1) $\al_k(\la_t) < \al_{k+1}(\la_t)$

\no(P2) $k-1 < \al_k(\la_t) < k+1$

\no and $\al_k(0)=k$, $\al_k(1)=\al_k$.   (It was conjectured that (P2) is redundant.)
Equivalently (see section 4 of \cite{TrWi98}), 
$z_1=e^{\pi\i\al_1/N},\dots,z_N=e^{\pi\i\al_N/N}$ are the roots of
the polynomial
\begin{equation}\label{twpoly}
z^{2N} + 2\pi\ii\la \sum_{j=1}^{N-1} c_j \om_j^{-1} z^{2j-N}  - 1,
\end{equation}
subject again to (P1),(P2).  

Now let us consider the case $N=4$.  We have $q_1+q_2=0$, $q_3+q_4=0$,
and also $c_1 \om_1^{-1}=-c_1\ii$,   $c_2 \om_2^{-1}=-c_2$,   $c_3 \om_3^{-1}=c_1$.
The system (\ref{twost}) coincides with our system (\ref{ost}) if we take
$x=t$, $n=3$, and either

(I) $2w_0=q_2$, $2w_1=q_3$\  \ i.e.\ $\ga_0=2(\al_2-2)$, $\ga_1=2(\al_3-3)$

\no or

(II) $2w_0=q_4$, $2w_1=q_1$\  \  i.e.\ $\ga_0=2(\al_4-4)$, $\ga_1=2(\al_1-1)$.

\no It follows from this that the roots 
$\mu=  
e^{\pm\tfrac\pi4 {\scriptstyle (\ga_0+1)} },
e^{\pm\tfrac\pi4 {\scriptstyle(\ga_1+3)} }
$
of our polynomial 
\[
p(\mu)=\mu^4 + s_1^\R \mu^3 - s_2^\R \mu^2 + s_1^\R \mu + 1
\]
from section \ref{vanishing} are related to the roots 
$z^2=e^{\pi\i\al_k/2}$ ($1\le k\le 4$)
 of the polynomial (\ref{twpoly}) by
\[
\text{ (I) $\mu=z^2e^{-3\pi\i/4}$ or (II) $\mu=z^2e^{\pi\i/4}$.}
\]
Comparing $p$ with (\ref{twpoly}), we obtain

(I) $s_1^\R=-2\pi\ii \la c_1 e^{\pi\i/4}$,\quad $s_2^\R=2\pi\ii \la c_2 e^{\pi\i/2}$

\no or

(II) $s_1^\R=2\pi\ii \la c_1 e^{\pi\i/4}$,\quad $s_2^\R=2\pi\ii \la c_2 e^{\pi\i/2}$.

\no i.e.\ the parameters $c_1,c_2$ are essentially our Stokes parameters.  
The involution $(w_0,w_1)\mapsto (-w_1,-w_0)$ relates (I) and (II); it reverses the sign of $s_1^\R$.

Our Theorem \ref{linalg} (which relies on Theorem A of \cite{GuItLiXX}) allows us to verify the above assumption concerning the existence of the continuous path $\la_t$.  Namely, the region of smooth solutions is given by the conditions $\ga_{i+1}-\ga_i\ge -2$, which are equivalent to 
the conditions $\al_{i+1}-\al_i\ge 0$. The linear path $\la_t=t\la$ has the required properties:  this corresponds to a path in the $(s_1^\R,s_2^\R)$-plane from $(0,0)$ to the point $(s_1^\R,s_2^\R)$, and it lies entirely within the region (a) of smooth solutions, as this region is star-shaped (Fig.\ \ref{results4}).

We remark that the path $\la_t$ is needed only to explain how to recover $\ga_0,\ga_1$ from $s_1^\R,s_2^\R$.  Theorem \ref{alldata} expresses $s_1^\R,s_2^\R$ explicitly in terms of $\ga_0,\ga_1$, but to go in the opposite direction it is necessary to solve the polynomial equation $p(\mu)=0$ to obtain 
$e^{\pm\tfrac\pi4 {\scriptstyle (\ga_0+1)} },
e^{\pm\tfrac\pi4 {\scriptstyle(\ga_1+3)} }$, then specify the correct point in the inverse image of the covering map --- for example by specifying a path to be lifted.

Of course (P1) gives exactly the interior of the region.  Condition (P2) should be deleted, as it corresponds to the region $-2<\ga_0,\ga_1<2$, which (if imposed) would give only a proper subset of the interior.

\section{Appendix A: Various matrices}

\no{\em Frequently used constants:}

\medskip

\[
\Pi=
\bp
  & \!1\! & & \\
 & & \!1\! & \\
  & & & \!1\\
1\!   & & &
   \ep
\ \ \ 
\Om=
\bp
1 & 1& 1 & 1\\
1 & \om & \om^2 & \om^3 \\
1 & \om^2 & \om^4 & \om^6 \\
1 & \om^3 & \om^6 & \om^9 
\ep
\ \ \ 
d_4=
\bp
1 & & &\\
 & \!\om\! & & \\
  & & \!\om^2\! & \\
  & & & \!\!\om^3
\ep
\]

\[
\De =
\bp
 & & & \!1\\
 & & \!1\! & \\
 & \!1\! & &\\
 1\! & & &
\ep
\ \ \ 
C=\bp
1\! & & & \\
 & & & \!1\\
 & & \!1\! & \\
 & \!1\! & &
\ep
\]

\no{\em Useful identities:}

\medskip

(F1) $\bar\Om \Om=4I$ 
\quad 
(F2) $\Pi\Om=\Om d_4$
\quad
(F3) $\Pi C =\De = C\Pi^{-1}$

(F4) $d_4 C  d_4 = C$
\quad
(F5) $\Om d_4 \Om=4\De$, $\Om\De\Om=4d_4^{-1}$

(F6) $C=\Om\bar\Om^{-1}=\tfrac14\Om^2=4\Om^{-2}$
\quad
(F7) $\Pi d_4=\om d_4\Pi$, $\Pi^2 d_4=-d_4\Pi^2$.

\no{\em The matrices $\Qi_k$:}

\[
\Qi_{1}=
\bp
1 & & & \\
\om^{\scriptstyle\frac32} s_1^{\R} & 1 & & \\
 & & 1 & \om^{\scriptstyle\frac12} s_1^{\R} \\
  & & & 1
\ep
\ \ \ \ 
\Qi_{1\scriptstyle\frac14}=
\bp
1 & & & \\
 & 1 & &\om^{3} s_2^{\R} \\
 & & 1 &  \\
  & & & 1
\ep
\]

\[
\Qi_{1\scriptstyle\frac12}=
\bp
1 & & & \om^{\scriptstyle\frac32} s_1^{\R}\\
 & 1 &\om^{\scriptstyle\frac12} s_1^{\R} & \\
 & & 1 &  \\
  & & & 1
\ep
\ \ \ \ 
\Qi_{1\scriptstyle\frac34}=
\bp
1 & &\om^{3} s_2^{\R} & \\
 & 1 & & \\
 & & 1 &  \\
  & & & 1
\ep
\]

\[
\Qi_{2}=
\bp
1 &\om^{\scriptstyle\frac12} s_1^{\R} & & \\
 & 1 & & \\
 & & 1 &  \\
  & &\om^{\scriptstyle\frac32} s_1^{\R} & 1
\ep
\ \ \ \ 
\Qi_{2\scriptstyle\frac14}=
\bp
1 & & & \\
 & 1 & & \\
 & & 1 &  \\
  & \om^{3} s_2^{\R}& & 1
\ep
\]

\[
\Qi_{2\scriptstyle\frac12}=
\bp
1 & & & \\
 & 1 & & \\
 &\om^{\scriptstyle\frac32} s_1^{\R} & 1 &  \\
\om^{\scriptstyle\frac12} s_1^{\R}  & & & 1
\ep
\ \ \ \ 
\Qi_{2\scriptstyle\frac34}=
\bp
1 & & & \\
 & 1 & & \\
\om^{3} s_2^{\R} & & 1 &  \\
  & & & 1
\ep
\]

\section{Appendix B: Summary of results for other cases}

Equations (\ref{ost}),(\ref{as}) depend on two integers $n,l$.  There are precisely ten cases where  $w_0,\dots,w_n$ reduce to two unknown functions 
(see section 2 of \cite{GuItLiXX}).  Then (\ref{ost}),(\ref{as}) reduce to a system of the form
\begin{equation*}
\begin{cases}
u_{z\zbar}&= \ e^{au} - e^{v-u} 
\\
v_{z\zbar}&= \ e^{v-u} - e^{-bv}
\end{cases}
\end{equation*}
where $a,b\in\{1,2\}$.   In this article we have investigated only case 4a.  To extend these results to the remaining nine cases, it suffices to treat case 5a and case 6a, as the other cases are easily related to these.  In this appendix we summarize the results for these two cases. Notation not explained here can be found in \cite{GuItLiXX}.

\no{\em Case 5a:}

$\om=e^{2\pi\i/5}$

$\dz=d_5^3=\di^{-1}$

$\tPi=\om^2 \Pi$

$S S^{-t} =  (\tQi_1 \tQi_{1\frac15}\tPi)^5
=  (\tQi_1 \tQi_{1\frac15} \Pi)^5$

\no The characteristic polynomial of $\tQi_1 \tQi_{1\frac15} \Pi$
is 
\begin{align*}
p(\mu)
&=-(\mu^5 - s_1^\R \mu^4 - s_2^\R \mu^3 + s_2^\R \mu^2 
+ s_1^\R \mu - 1)
\\
&=(1-\mu)(\mu^4 -(s_1^\R - 1)\mu^3 + (1-s_1^\R-s_2^\R)\mu^2 - (s_1^\R-1)\mu + 1).
\end{align*}
The identity
 \[
    X^5 + I    =
    (X-\om^{\frac12}) (X-\om^{\frac32}) (X-\om^{\frac52}) 
    (X-\om^{\frac72}) (X-\om^{\frac92}   )
 \]
 shows that
 \begin{align*}
 \det\left(   (\tQi_1 \tQi_{1\frac15} \Pi)^5 + I   \right)
 &=
 p(\om^{\frac12})p(\om^{\frac32})p(\om^{\frac52})p(\om^{\frac72})
 p(\om^{\frac92})
 \\
 &=
 p(\om^{\frac12})^2p(\om^{\frac32})^2p(\om^{\frac52}).
\end{align*}
As in Theorem \ref{linalg}, we see that the region 
where $S^{-1}\!+\!S^{-t}>0$ is given by 
$p(\om^{\frac12})>0, p(\om^{\frac32})>0, p(\om^{\frac52})>0$, i.e.\ 
\[
 2+ 
 {\scriptstyle  \frac{-1-\sqrt 5}{2} } s_1^\R + 
 {\scriptstyle  \frac{1-\sqrt 5}{2}  } s_2^\R>0,
 2+ 
{\scriptstyle  \frac{-1+\sqrt 5}{2}  } s_1^\R + 
{\scriptstyle  \frac{1+\sqrt 5}{2}  } s_2^\R>0, 
2+2s_1^\R-2s_2^\R>0.
 \]
This is the shaded region  of Fig.\ \ref{results5}. 

\begin{figure}[h]
\begin{center}
\includegraphics[scale=0.4, trim= 40  120  0  120]{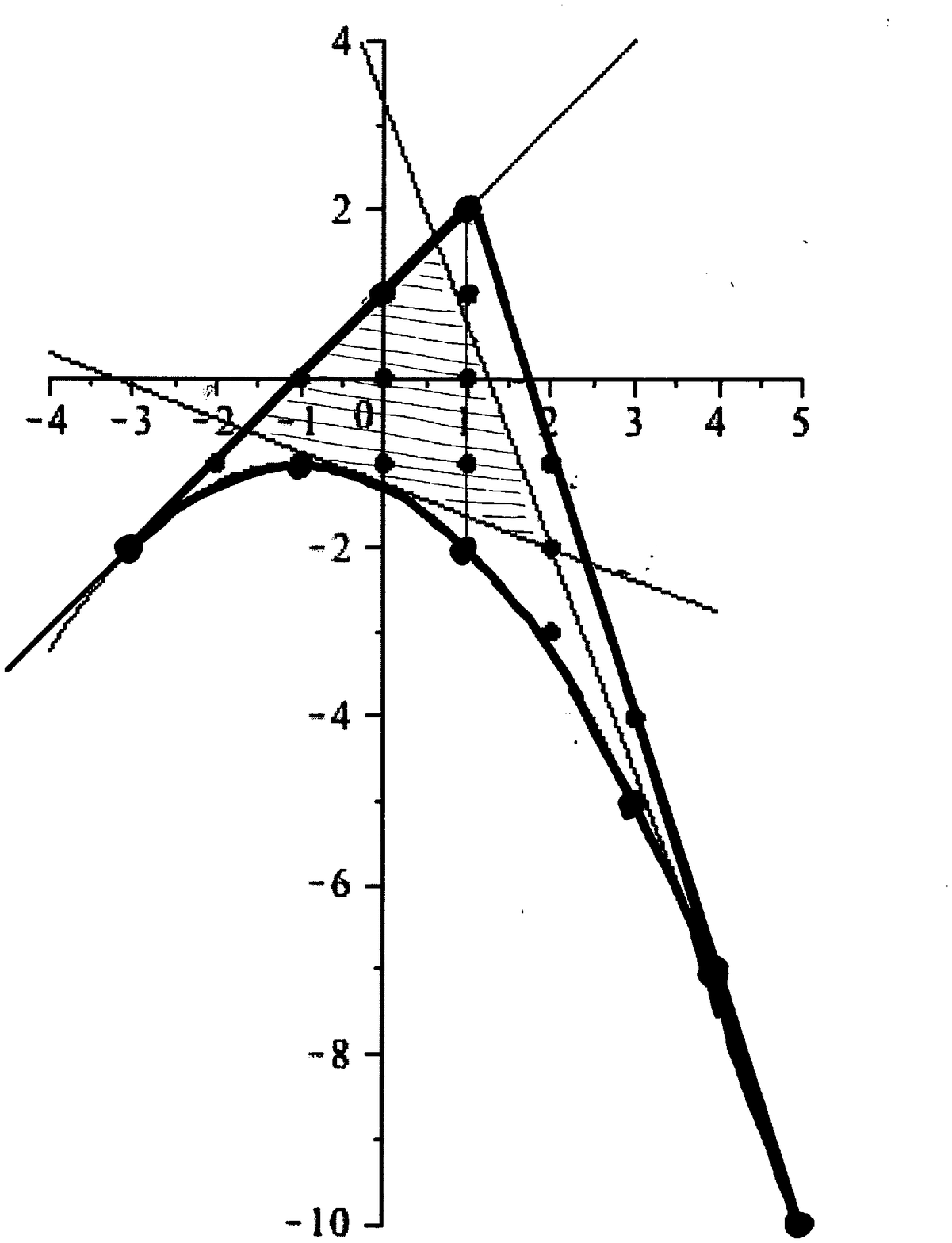}
\end{center}
\caption{Solutions of the tt*-Toda equations (case 5a).}\label{results5}
\end{figure}

The region of smooth solutions (from Theorem A of \cite{GuItLiXX}) is the region where all roots of
$P(x)=x^2 +(1-s_1^\R)x-(1+s_1^\R+s_2^\R)$ lie in the interval $[-2,2]$.  As in Theorem \ref{linalg}, we see that this is given by 
\[
(s_1^\R)^2+2s_1^\R+4s_2^\R+5\ge0,\ 
5-3s_1^\R-s_2^\R\ge0,\ 
1+s_1^\R-s_2^\R\ge0.
\]
This is the (closed) region bounded by heavy lines in Fig.\ \ref{results5}. The dots indicate the integral points in this region.  

The roots of $p$ are $1, e^{\pm\i \th_1}, e^{\pm\i \th_2}$ with 
$\th_1=\tfrac\pi5 { (\ga_0+6)}$, 
$\th_2=\tfrac\pi5 { (\ga_1+8)}$
and $0\le \th_1\le \th_2 \le\pi$.  
The region where $S^{-1}\!+\!S^{-t}>0$ is characterized by the interlacing condition
\[
0<\tfrac\pi5 < \th_1 < \tfrac{3\pi}5 < \th_2 < \pi.   
\]This follows from the
explicit formulae
\begin{align*}
 p(\om^{\frac12})&=8\cos \tfrac{2\pi}5 (\cos \th_1 - \cos \tfrac{\pi}5)(\cos \th_2 - \cos \tfrac{\pi}5)
 \\
 p(\om^{\frac32})&=-8\cos \tfrac{2\pi}5 (\cos \th_1 - \cos \tfrac{3\pi}5)(\cos \th_2 - \cos \tfrac{3\pi}5)
 \\
 p(\om^{\frac52})&=8(\cos \th_1 + 1)(\cos \th_2 + 1)\  (\ge 0)
\end{align*}
as the conditions
$p(\om^{\frac12})>0, p(\om^{\frac32})>0$
mean that $\th_1,\th_2$ lie on the same side of $\tfrac{\pi}5$
but opposite sides of $\tfrac{3\pi}5$.

\no{\em Case 6a:}

$\om=e^{2\pi\i/6}$

$\dz=d_6=\di^{-1}$

$\tPi=\om^{-1}\Pi$

$S S^{-t} =  (\tQi_1 \tQi_{1\frac16}\tPi)^6
=  (\tQi_1 \tQi_{1\frac16} \Pi)^6$

\no The characteristic polynomial of $\tQi_1 \tQi_{1\frac16} \Pi$
is 
\begin{align*}
p(\mu)
&=\mu^6 + s_1^\R \mu^5 - s_2^\R \mu^4 + s_2^\R \mu^2 
- s_1^\R \mu - 1
\\
&=(\mu-1)(\mu+1)(\mu^4 + s_1^\R\mu^3 +(1- s_2^\R)\mu^2 + s_1^\R\mu + 1).
\end{align*}
The identity
 \[
    X^6 + I    =
(X-\om^{\frac12}) (X-\om^{\frac32}) (X-\om^{\frac52}) (X-\om^{\frac72}) 
(X-\om^{\frac92}) (X-\om^{\frac{11}2})
 \]
 shows that
 \begin{align*}
 \det\left(  (\tQi_1 \tQi_{1\frac16} \Pi)^6 - I \right)
 &=
 p(\om^{\frac12})p(\om^{\frac32})p(\om^{\frac52})p(\om^{\frac72})
 p(\om^{\frac92})p(\om^{\frac{11}2})
 \\
 &=
 p(\om^{\frac12})^2p(\om^{\frac32})^2p(\om^{\frac52})^2.
 \end{align*}
As in Theorem \ref{linalg}, we see that the region 
where $S^{-1}\!+\!S^{-t}>0$ is given by 
$p(\om^{\frac12})<0, p(\om^{\frac32})<0, p(\om^{\frac52})<0$, i.e.\ 
\[
2+\sqrt3 s_1^\R-s_2^\R>0,\ 
2+2s_2^\R>0,\ 
2-\sqrt3 s_1^\R-s_2^\R>0.
\]
This is the shaded region  of Fig.\ \ref{results6}. 

\begin{figure}[h]
\begin{center}
\includegraphics[scale=0.4, trim= 40  120  0  120]{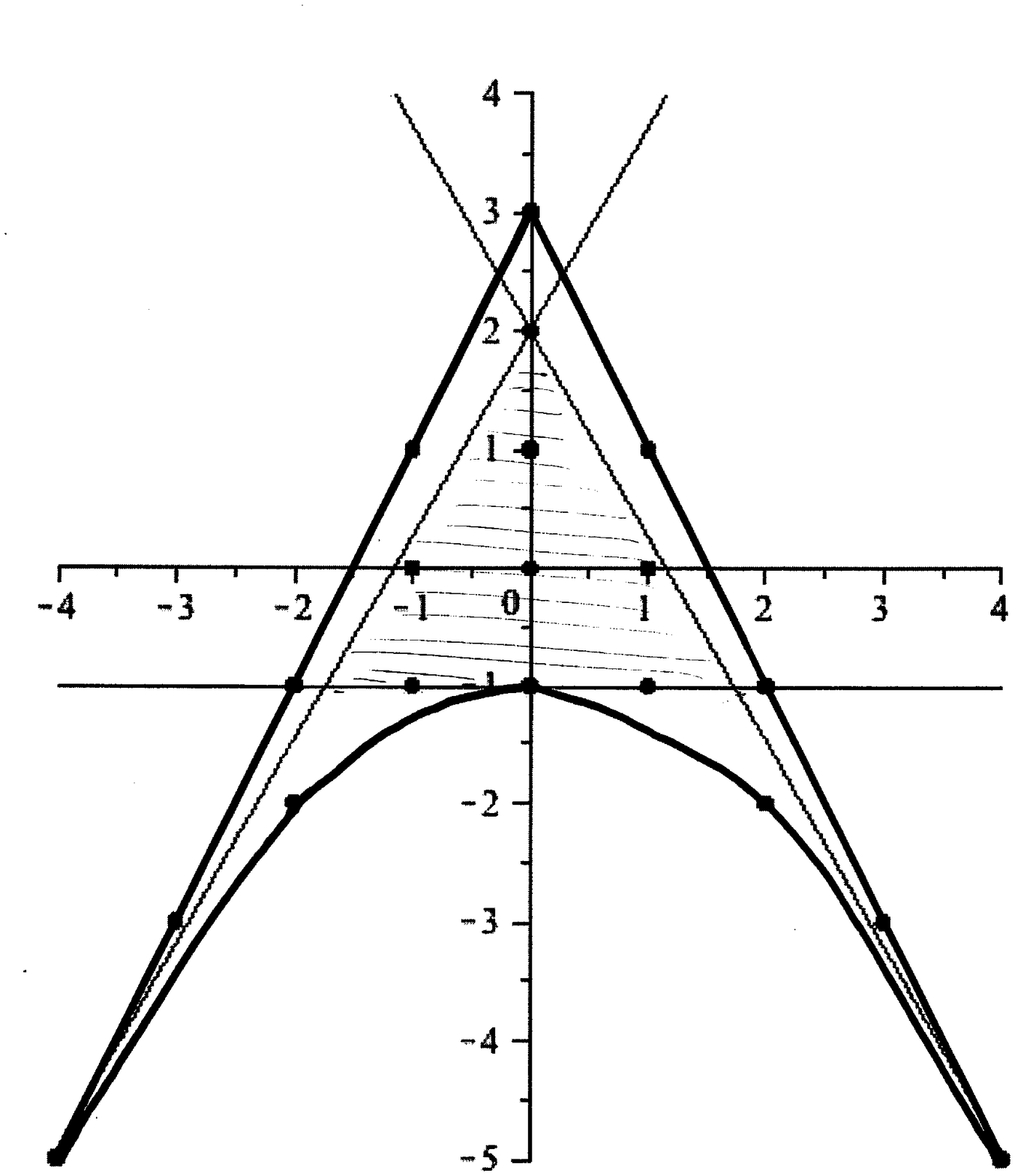}
\end{center}
\caption{Solutions of the tt*-Toda equations (case 6a).}\label{results6}
\end{figure}

The region of smooth solutions (from Theorem A of \cite{GuItLiXX}) is the region where all roots of
$P(x)=x^2 +s_1^\R x-(1+s_2^\R)$ lie in the interval $[-2,2]$.  As in Theorem \ref{linalg}, we see that this is given by 
\[
(s_1^\R)^2+4s_2^\R+4\ge0,\ 
3-2s_1^\R-s_2^\R\ge0,\ 
3+2s_1^\R-s_2^\R\ge0.
\]
This is the (closed) region bounded by heavy lines in Fig.\ \ref{results6}. The dots indicate the integral points in this region. 

The roots of $p$ are $\pm 1, e^{\pm\i \th_1}, e^{\pm\i \th_2}$ with 
$\th_1=\tfrac\pi6 { (\ga_0+2)}$, 
$\th_2=\tfrac\pi6 { (\ga_1+4)}$
and $0\le \th_1\le \th_2 \le\pi$.  
The region where $S^{-1}\!+\!S^{-t}>0$ is characterized by the interlacing condition
\[
0<\tfrac\pi6 < \th_1 < \tfrac{3\pi}6 < \th_2  < \tfrac{5\pi}6 < \pi.
\]
This follows from the
explicit formulae
\begin{align*}
 p(\om^{\frac12})&=-4(\cos \th_1 - \cos \tfrac{\pi}6)(\cos \th_2 - \cos \tfrac{\pi}6)
 \\
 p(\om^{\frac32})&=8\cos \th_1\cos \th_2
 \\
 p(\om^{\frac52})&=-4(\cos \th_1 - \cos \tfrac{5\pi}6)(\cos \th_2 - \cos \tfrac{5\pi}6)
\end{align*}
as the conditions
$p(\om^{\frac12})<0, p(\om^{\frac32})<0, p(\om^{\frac52})<0$
mean that $\th_1,\th_2$ lie on the same side of $\tfrac{\pi}6$,
opposite sides of $\tfrac{3\pi}6$, and the same side of $\tfrac{5\pi}6$.

\section{Appendix C: Asymptotics of radial solutions}

The purpose of this section is to prove Theorems \ref{theorem1} and \ref{theorem2} below.  These results were used in sections \ref{asymptotic} and \ref{vanishing} in order to relate the solutions obtained from the Riemann-Hilbert approach with the solutions obtained earlier in \cite{GuLiXX},\cite{GuItLiXX}.   While this method was primarily a matter of convenience,  the results themselves are of independent interest, as they show that Theorem A of \cite{GuItLiXX} accounts for all radial solutions which are smooth on $\C^\ast$. 

For $a,b>0$, let us consider the equations
\begin{align}
(u_1)_{z\zbar} &=e^{au_1} - e^{u_2-u_1}\label{equation1}\\
(u_2)_{z\zbar} &=e^{u_2-u_1} - e^{-bu_2} \label{equation2}
\end{align}
(this is system (2.1) of section 2 of \cite{GuItLiXX}; the case 4a considered in this article is the case $a=b=2$).  

In this section we sometimes write $x \ (=\text{Re} z + \ii \text{Im} z)\in\R^2$ instead of $z\in\C$, and use the notation $\De f=4f_{z\zbar}$.  If $f$ depends only on $r=\vert x-p\vert$ for some fixed $p$, we have $\De f=f^\prr + \tfrac1r f^\pr$, where prime denotes derivative with respect to $r$.  We also use the notation $B(p,a)=\{ x\in\R^2 \st \vert x-p\vert < a \}$.

\begin{lemma}\label{lemma1} 
Assume $u_1,u_2$ are smooth solutions of (\ref{equation1}), (\ref{equation2}) on $\C^\ast$.  Then there exist $\be_i>0$
such that 
\[
\text{ $\vert u_i(x)\vert \le -\be_i \log \vert x\vert $ ($i=1,2$) }
\]
for small $\vert x\vert$.  
\end{lemma}

\begin{proof}  The proof is based on the method of \cite{ChLi97}.   

Let us consider any $x_0$ with
$0<\vert x_0\vert < \tfrac14$.   Let $r_0=  \tfrac12\vert x_0\vert$.    Consider any $y_0$ with 
$\vert x_0-y_0\vert = \tfrac14 r_0$.  In the following argument we fix $y_0$.  Writing $r=\vert x - y_0\vert$,  we introduce a function
\[
w(r)=\log \tfrac1{r(r_0-r)},\quad 0<r<r_0.
\]
This satisfies
$w(r)\to+\infty$ as $r\to 0$ or $r\to r_0$, and
$
w^\prr +\tfrac1r w^\pr = \tfrac{r_0}{r(r_0-r)^2}.
$

For $i=1,2$ let $w_i(x)=\be_i w(r)$, where $\be_1,\be_2$ are positive constants to be chosen shortly.
For $x\in B(y_0,r_0)$, we have
\begin{equation}\label{lhs}
\De w_i = \tfrac{\be_i r_0}{r(r_0-r)^2}\le \tfrac{1}{r^2(r_0-r)^2}
\end{equation}
if $r_0$ is small enough.

On the other hand
\begin{equation}\label{rhs1}
e^{aw_1} - e^{w_2-w_1}=
\left( \tfrac1{(r_0-r)r} \right)^{a\be_1} 
-
\left( \tfrac1{(r_0-r)r} \right)^{\be_2-\be_1}
\end{equation}
\begin{equation}\label{rhs2}
e^{w_2-w_1} - e^{-bw_2}=
\left( \tfrac1{(r_0-r)r} \right)^{\be_2-\be_1} 
-
\left( \tfrac1{(r_0-r)r} \right)^{-b\be_2}.
\end{equation}
Therefore, on $B(y_0,r_0)$, we obtain
\begin{align}
(w_1)_{z\zbar}&\le e^{aw_1} - e^{w_2-w_1}\label{sub1}
\\
(w_2)_{z\zbar}&\le e^{w_2-w_1} - e^{-bw_2}\label{sub2}
\end{align}
if we choose $a\be_1>2$ and $\be_2-\be_1=2$ and if $r_0$
is sufficiently small.  
We shall prove next that 
\begin{equation}\label{wsubsolution}
\text{$u_i\le w_i$ \ on $B(y_0,r_0)$  \ \  ($i=1,2$).}
\end{equation}

If $u_1\le w_1$ does not hold, then there exists some $z$ such that
\[
(u_1-w_1)(z)=\max_{\vert x-y_0\vert \le r_0}  (u_1-w_1)(x) > 0.
\]
We have $0<\vert z-y_0\vert < r_0$ 
because $w_1(x)\to+\infty$ as $x\to y_0$.
Hence, using the maximum principle, and inequality (\ref{sub1}), we obtain
\[
0\ge (u_1-w_1)_{z\zbar}(z)
\ge e^{au_1(z)} -  e^{aw_1(z)} -
( e^{(u_2-u_1)(z)} - e^{(w_2-w_1)(z)} ).
\]
Since $e^{au_1(z)} - e^{aw_1(z)}>0$,  we obtain $ e^{(u_2-u_1)(z)} - e^{(w_2-w_1)(z)} >0$,
hence
\begin{equation}\label{2greater than1}
(u_2-w_2)(z) > (u_1-w_1)(z) >0.
\end{equation}
Next, as $(u_2-w_2)(z) >0$,  
there exists some $\tilde z$ such that
\[
(u_2-w_2)(\tilde z)=\max_{\vert x-y_0\vert \le r_0}  (u_2-w_2)(x) > 0.
\]
The maximum principle, and inequality (\ref{sub2}), give
\[
0\ge (u_2-w_2)_{z\zbar}(\tilde z)
\ge e^{(u_2-u_1)(\tilde z)} -  e^{(w_2-w_1)(\tilde z)} -
( e^{-bu_2(\tilde z)} - e^{(-bw_2(\tilde z)} ).
\]
Since $e^{-bu_2(\tilde z)} - e^{(-bw_2(\tilde z)}<0$,  we obtain $ e^{(u_2-u_1)(\tilde z)} - e^{(w_2-w_1)(\tilde z)} <0$, i.e.\ 
$(u_2-w_2)(\tilde z) < (u_1-w_1)(\tilde z)$.
Thus
\begin{align*}
(u_2-w_2)(z)&\le (u_2-w_2)(\tilde z) \text{ by definition of $\tilde z$}
\\
&< (u_1-w_1)(\tilde z) 
\\
&\le (u_1-w_1)(z) \text{ by definition of $z$.}
\end{align*}
But this gives
\begin{equation}\label{1greater than2}
(u_2-w_2)(z) < (u_1-w_1)(z)
\end{equation}
which contradicts (\ref{2greater than1}).
We conclude that $u_1\le w_1$, as required.
A similar argument gives $u_2\le w_2$.   This establishes (\ref{wsubsolution}).

In particular we have $u_1(x_0) \le w_1(x_0) = 2\be_1 \log \tfrac1{\vert x_0\vert} + c_1$ and
$u_2(x_0) \le w_2(x_0) = 2\be_2 \log \tfrac1{\vert x_0\vert} + c_2$ for some constants $c_1,c_2$.

Finally, replacing $u_1,u_2$ by $-u_2,-u_1$, the same argument gives bounds for $-u_1,-u_2$.  This completes the proof of the lemma.
\end{proof}

\begin{remark}\label{largex}
If  $\vert x_0\vert >> 1$, we may take $r_0=1$ in the above argument,  which then gives
$\vert u_i(x)\vert \le c_i$ for large $x$, for some constants $c_1,c_2$.  
\end{remark}

From now on, we assume that $u_i$ is a radial solution, i.e.\ $u_i=u_i(r)$ where $r=\vert x\vert$.  
Let $B_0=\{ x\in \R^2 \st 0<\vert x\vert \le 1 \}$.

\begin{lemma}\label{lemma2}    
Assume $u_1,u_2$ are smooth radial solutions of (\ref{equation1}), (\ref{equation2}) on $\C^\ast$. 
If one of $e^{au_1},e^{u_2-u_1},e^{-bu_2}$ is in $L^1(B_0)$, then
all three are, and the limits
\[
\lim_{r\to 0} \frac{u_i(r)}{\log r}\quad (i=1,2)
\]
exist.
\end{lemma}

\begin{proof}  Assume that  $e^{au_1}\in L^1(B_0)$.   By Lemma \ref{lemma1}, there exists a sequence
$r_m\to 0$ in the interval $(0,1)$ such that $r_m u_1^\pr(r_m)$ remains bounded.  
Integration of equation (\ref{equation1}) gives
\[
 u_1^\pr(1)-r_m u_1^\pr(r_m) 
 =  4\int_{r_m}^1 e^{au_1} - e^{u_2-u_1}  \ rdr.
\]
It follows
that  $\lim_{m\to+\infty} \int_{r_m}^1 e^{u_2-u_1}\ rdr$ exists, hence 
$\lim_{r\to0} \int_{r}^1 e^{u_2-u_1} \ rdr$ exists,  i.e.\ $e^{u_2-u_1}\in L^1(B_0)$.  Integrating now from $r$ to $1$,   we see that 
$\lim_{r\to0} {ru_1(r)}$ exists, hence
$\lim_{r\to0} {u_1(r)}/{\log r}$ exists, as required.  From equation (\ref{equation2}) we deduce
in a similar way that
$e^{-bu_2}\in L^1(B_0)$ and 
$\lim_{r\to0} {u_2(r)}/{\log r}$ exists.  This completes the proof when $e^{au_1}\in L^1(B_0)$.
If $e^{u_2-u_1}\in L^1(B_0)$ or $e^{-bu_2}\in L^1(B_0)$, the proof is similar.
\end{proof}

\begin{theorem}\label{theorem1}
Assume $u_1,u_2$ are smooth radial solutions of (\ref{equation1}), (\ref{equation2}) on $\C^\ast$. 
Then the limits
\[
\lim_{r\to 0} \frac{u_i(r)}{\log r}\quad (i=1,2)
\]
exist.
\end{theorem}

\begin{proof} $ $

\no{\em Case I: Either $u_1(r)$ or $u_2(r)$ has infinitely many local maxima or minima $r_m$ ($m\in\N$) with $\lim_{m\to+\infty} r_m=0$.}

If $u_1$ satisfies this condition, we shall prove that $e^{au_1}$ is in $L^1(B_0)$, then apply Lemma \ref{lemma2} to obtain Theorem \ref{theorem1} in this case. 
Without loss of generality we may assume that
\[
0< \cdots < r_{m+1} <  r_m < \cdots < r_1\le 1
\]
where the $r_{2m-1}$ are local minima of $u_1$ and the $r_{2m}$ are local maxima of $u_1$.   By the maximum principle we have
\begin{align}
au_1(r_{2m-1}) &\ge (u_2-u_1)(r_{2m-1}), \text{ and} \label{min1}\\
au_1(r_{2m}) &\le (u_2-u_1)(r_{2m})\label{max1}.
\end{align}
For the intervals
\[
I_m=[r_{2m},r_{2m-2}],\quad I^\ast_m=[r_{2m+1},r_{2m-1}],
\]
there exist $\s\in I_m$, $\s^\ast\in I^\ast_m$, such that
\[
(u_1-u_2)(\s)=\max_{r\in I_m} (u_1-u_2)(r), \quad
(u_2-u_1)(\s^\ast)=\max_{r\in I_m^\ast} (u_2-u_1)(r).
\]
Note that by (\ref{max1}) 
\[
(u_1-u_2)(r_{2m}) \le -au_1(r_{2m}), \text{ and } 
(u_1-u_2)(r_{2m-2}) \le -au_1(r_{2m-2}),
\]
and by (\ref{min1})
\[
(u_1-u_2)(r_{2m-1}) \ge -au_1(r_{2m-1}) > \max\{-au_1(r_{2m-2}), -au_1(r_{2m})\}.
\]
It follows that $\s$ is an interior point of $I_m$.  A similar argument shows that  $\s^\ast$ is an interior point of $I^\ast_m$.

By the maximum principle for 
$(u_2-u_1)_{z\zbar}=2e^{u_2-u_1}- e^{-bu_2}-e^{au_1}$
(the difference of equations
(\ref{equation1}), (\ref{equation2})),  we have
\begin{align}
2e^{(u_2-u_1)(\s)} &\ge e^{au_1(\s)}+ e^{-bu_2(\s)}, \text{ and}  \label{max12}\\
2e^{(u_2-u_1)(\s^\ast)} &\le e^{au_1(\s^\ast)}+ e^{-bu_2(\s^\ast)}  \label{max21}.
\end{align}

First we shall establish a lower bound for $u_1$: 

\no{\em Assertion 1.\  For all $r\in(0,c]$, with $c$ sufficiently small, we have 
$u_1(r)\ge -\tfrac1a \max\{ \tfrac1a\log2,  \tfrac1b\log2 \}$. } 

\no We may assume
\begin{equation}\label{u1oddpos}
u_1(r_{2m-1})<0
\end{equation}
for a sequence of $m$ such that $m\to+\infty$ (otherwise $u_1(r)\ge 0$ for $r$ sufficiently small, and the lower bound holds already).  Hence, for such $m$,
\begin{equation}\label{u12oddpos}
(u_1-u_2)(\s) \ge (u_1-u_2)(r_{2m-1}) \ge -au_1(r_{2m-1}) > 0,
\end{equation}
by definition of $\s$ and by using (\ref{min1}), (\ref{u1oddpos}).

We claim that $(u_1-u_2)(\s)\le \max\{ \tfrac1a\log2,  \tfrac1b\log2 \}$.  
First we note that
\begin{equation}\label{frommax12}
2>2e^{(u_2-u_1)(\s)} \ge e^{au_1(\s)}+ e^{-bu_2(\s)}\ge  e^{-bu_2(\s)},
\end{equation}
by (\ref{u12oddpos}) and  (\ref{max12}).   

If $u_1(\s)\le 0$,  then
\begin{equation}\label{firstbound12}
(u_1-u_2)(\s)\le -u_2(\s)\le \tfrac1b \log 2
\end{equation}
by (\ref{frommax12}).  

If $u_1(\s)> 0$,  we must have
\begin{equation}\label{u2pos}
u_2(\s)>0.
\end{equation}
Namely, if $u_2(\s)\le0$, (\ref{frommax12}) gives 
$2>2e^{(u_2-u_1)(\s)} \ge e^{au_1(\s)}+ e^{-bu_2(\s)}
\ge 1+ 1=  2$,
which is a contradiction.  Thus (\ref{u2pos}) holds, and hence we have
\begin{equation}\label{u12lessthanu1}
(u_1-u_2)(\s)<u_1(\s).
\end{equation}
On the other hand,   (\ref{max12}) and  (\ref{u12oddpos}) give the following
upper bound for $u_1(\s)$:
\[
e^{au_1(\s)}\le 2e^{(u_2-u_1)(\s)} -  e^{-bu_2(\s)} \le 2e^{(u_2-u_1)(\s)} \le 2.
\]
We obtain 
\begin{equation}\label{secondbound12}
(u_1-u_2)(\s)\le u_1(\s) \le \tfrac1a\log 2.
\end{equation}
Combining (\ref{firstbound12}) and  (\ref{secondbound12}) we obtain 
$(u_1-u_2)(\s)\le \max\{ \tfrac1a\log2,  \tfrac1b\log2 \}$, as claimed.

Now we can complete the proof of Assertion 1.  We have
\begin{align*}
-au_1(r_{2m-1})&\le (u_1-u_2)(r_{2m-1})  \text{ by (\ref{min1}) }
\\
&\le (u_1-u_2)(\s)  \text{ by definition of $\s$ }
\\
&\le \max\{ \tfrac1a\log2,  \tfrac1b\log2 \}  \text{ by the claim above, }
\end{align*}
i.e.\ $u_1(r_{2m-1}) \ge -\tfrac1a \max\{ \tfrac1a\log2,  \tfrac1b\log2 \}$.  
Assertion 1 follows immediately from this.

Next we use Assertion 1 to establish an upper bound for $u_1$.   By Lemma \ref{lemma2}, this 
will complete the proof of the theorem in Case I.  

\no{\em Assertion 2.\  
For all $r\in(0,c]$, with $c$ sufficiently small, we have 
$u_1(r)\le \tfrac1{a^2} \max\{ \tfrac1a\log2,  \tfrac1b\log2 \}$. } 

\no We may assume
\begin{equation}\label{u1evenpos}
u_1(r_{2m})>0
\end{equation}
for a sequence of $m$ such that $m\to+\infty$ (otherwise $u_1(r)\le 0$ for $r$ sufficiently small, and the upper bound holds already).  Hence, for such $m$,
\begin{equation}\label{u21evenpos}
(u_2-u_1)(\s^\ast) \ge (u_2-u_1)(r_{2m}) \ge au_1(r_{2m}) > 0,
\end{equation}
by definition of $\s^\ast$ and by using (\ref{max1}), (\ref{u1evenpos}).

We shall show that $u_2(\s^\ast)<0$.  First note that
\begin{align*}
au_1(\s^\ast) &\le au_1(r_{2m})  \text{ by definition of $r_{2m}$ }
\\
&\le (u_2-u_1)(r_{2m})  \text{ by (\ref{max1}) }
\\
&\le (u_2-u_1)(\s^\ast) \text{ by definition of $\s^\ast$. }
\end{align*}
Then we have
\begin{align*}
1&< e^{(u_2-u_1)(\s^\ast)}  \text{ by (\ref{u21evenpos}) }
\\
&\le 2e^{(u_2-u_1)(\s^\ast)} - e^{au_1(\s^\ast)}  \text{ by the previous paragraph }
\\
&\le e^{-bu_2(\s^\ast)}  \text{ by (\ref{max21}). }
\end{align*}
This shows that $u_2(\s^\ast)<0$.

Now we can complete the proof of Assertion 2.  We have
\begin{align*}
au_1(r_{2m})&\le (u_2-u_1)(r_{2m})  \text{ by (\ref{max1}) }
\\
&\le (u_2-u_1)(\s^\ast)  \text{ by definition of $\s^\ast$ }
\\
&< -u_1(\s^\ast) \text{ as $u_2(\s^\ast)<0$ }
\\
&\le \tfrac1a \max\{ \tfrac1a\log2,  \tfrac1b\log2 \}  \text{ by Assertion 1, }
\end{align*}
i.e.\ $u_1(r_{2m}) \le \tfrac1{a^2} \max\{ \tfrac1a\log2,  \tfrac1b\log2 \}$.  
Assertion 2, and hence the proof of the theorem in Case I,  follows immediately from this.

\no{\em Case II: Both $u_1(r)$ and $u_2(r)$ are monotone as $r\to 0$.}

If $u_1^\pr(r)>0$ as $r\to 0$ then $e^{au_1(r)}$ is bounded as $r\to 0$, 
so $e^{au_1} \in L^1(B_0)$.  Similarly, if $u_2^\pr(r)<0$ as $r\to 0$ then $e^{-bu_2} \in L^1(B_0)$.
It remains to consider the case where $u_1^\pr(r)<0$, $u_2^\pr(r)>0$.  But then 
$(u_2-u_1)^\pr(r)>0$ as $r\to 0$, so $e^{u_2-u_1} \in L^1(B_0)$.   Thus, Lemma \ref{lemma2}
completes the proof.  
\end{proof}

\begin{theorem}\label{theorem2}
Assume $u_1,u_2$ are smooth radial solutions of (\ref{equation1}), (\ref{equation2}) on $\C^\ast$. 
Then
$
\lim_{r\to +\infty} u_i(r) = 0 \quad (i=1,2).
$
\end{theorem}

\begin{proof} 
By Remark \ref{largex}, we know that both $u_1(r)$ and $u_2(r)$ are bounded as $r\to+\infty$.
If there is a sequence $r_n\to+\infty$ with $(u_1(r_n),u_2(r_n))\to(c,d)\ne(0,0)$,
then there is a $\de_0>0$ such that
$( e^{a\xi} - e^{\eta-\xi},  e^{\eta-\xi} - e^{-b\eta} )\ne(0,0)$
for any $(\xi,\eta)$ with $\vert\xi-c\vert < \de_0$ and $\vert \eta-d\vert<\de_0$.
By standard estimates for linear elliptic p.d.e., we have $\vert\nabla u(x) \vert \le c$ for
some $c$, and for all $r=\vert x\vert \ge R$, where $R$ is large.  Therefore
there exists $r_0$ (independent of $\eta$) such that
\[
\vert  u_1(r) - u_1(r_n) \vert \le \tfrac{\de_0}2,\quad
\vert  u_2(r) - u_2(r_n) \vert \le \tfrac{\de_0}2
\]
for $r\in[r_n-r_0,r_n+r_0]$.  Thus,
\[
\vert  e^{au_1(r)} - e^{u_2(r) - u_1(r)} \vert \ge \eps_0 > 0, \text{ for all $r\in[r_n-r_0,r_n+r_0]$ }
\]
or
\[
\vert  e^{u_2(r) - u_1(r)} - e^{-b u_2(r)}  \vert \ge \eps_0 > 0, \text{ for all $r\in[r_n-r_0,r_n+r_0]$ }
\]
holds for some $\eps_0$ (independent of $\eta$).  Suppose that the first inequality holds.
Then
\begin{align*}
\eps_0(2r_n+r_0)r_0
&\le
\int_{r_n-r_0}^{r_n+r_0} 
\vert  e^{au_1(r)} - e^{u_2(r) - u_1(r)} \vert  \ rdr
\\
&= \left\vert
\int_{r_n-r_0}^{r_n+r_0} 
 e^{au_1(r)} - e^{u_2(r) - u_1(r)}\  rdr  
 \right\vert
 \le C
\end{align*}
which yields a contradiction as $r_n\to+\infty$.  
Thus the limit of  $(u_1,u_2)$ exists and is equal to $(0,0)$.
\end{proof}

{\em

\noindent
Department of Mathematics\newline
Faculty of Science and Engineering\newline
Waseda University\newline
3-4-1 Okubo, Shinjuku, Tokyo 169-8555\newline
JAPAN

\noindent
Department of Mathematical Sciences\newline
Indiana University-Purdue University, Indianapolis\newline
402 N. Blackford St.\newline
Indianapolis, IN 46202-3267\newline
USA
   
\noindent
Taida Institute for Mathematical Sciences\newline
Center for Advanced Study in Theoretical Sciences  \newline
National Taiwan University \newline
Taipei 10617\newline
TAIWAN
}

\end{document}